\documentclass[11pt]{article}
\usepackage{amsmath, amsfonts, amsthm, amssymb, multicol}
\usepackage{graphicx}
\usepackage{float}
\usepackage{verbatim}

\usepackage[square,sort,comma,numbers]{natbib}
\usepackage{fancyhdr}
 
\numberwithin{equation}{section}

\newcommand{\diam}{\operatorname{diam}}
\newtheorem{thm}{Theorem}[section]
\newtheorem{lma}[thm]{Lemma}
\newtheorem{cor}[thm]{Corollary}
\newtheorem{defn}[thm]{Definition}

\newtheorem{prop}[thm]{Proposition}

\newtheorem{ques}[thm]{Question}

\renewcommand{\geq}{\geqslant}
\renewcommand{\leq}{\leqslant}
\renewcommand{\H}{\text{H}}

\title{ \vspace{-20mm}The Assouad dimensions of projections of planar sets}

\author{Jonathan M. Fraser and Tuomas Orponen}

\begin{document}

\date{}

\maketitle

\begin{abstract}
We consider the Assouad dimensions of orthogonal projections of planar sets onto lines. Our investigation covers both general and self-similar sets.

For general sets, the main result is the following: if a set in the plane has Assouad dimension $s \in [0,2]$, then the projections have Assouad dimension at least $\min\{1,s\}$ almost surely. Compared to the famous analogue for Hausdorff dimension -- namely \emph{Marstrand's Projection Theorem} -- a striking difference is that the words `at least' cannot be dispensed with: in fact, for many planar self-similar sets of dimension $s < 1$, we prove that the Assouad dimension of projections can attain both values $s$ and $1$ for a set of directions of positive measure.

For self-similar sets, our investigation splits naturally into two cases: when the group of rotations is discrete, and when it is dense.  In the `discrete rotations' case we prove the following dichotomy for any given projection: either the Hausdorff measure is positive in the Hausdorff dimension, in which case the Hausdorff and Assouad dimensions coincide; or the Hausdorff measure is zero in the Hausdorff dimension, in which case the Assouad dimension is equal to 1.  In the `dense rotations' case we prove that every projection has Assouad dimension equal to one, assuming that the planar set is not a singleton.
 
As another application of our results, we show that there is no \emph{Falconer's Theorem} for Assouad dimension.  More precisely, the Assouad dimension of a self-similar (or self-affine) set is not in general almost surely constant when one randomises the translation vectors.
\end{abstract}

\section{Introduction}

One of the most fundamental dimension theoretic questions in geometric measure theory is: \emph{how does dimension behave under orthogonal projection}?  This line of research began with the seminal paper of Marstrand from 1954 \cite{Marstrand}, influenced by some earlier work of Besicovitch.  One version of Marstrand's Projection Theorem states that if $F \subseteq \mathbb{R}^2$ is an analytic set with Hausdorff dimension $\dim_\text{H} F = s \in [0,2]$, then the Hausdorff dimension of the orthogonal projection of $F$ in almost every direction is $\min\{1, s\}$, which is as big as it can be.  Here `almost every' refers to Lebesgue measure on the interval $[0,\pi)$, with projections parameterised in the obvious way.  For the purposes of this paper, the important thing about this projection theorem is that, no matter what analytic set $F \subseteq \mathbb{R}^2$ one considers, the Hausdorff dimension of the projection $\pi F$ is \emph{almost surely constant}, i.e., takes the same value for almost every $\pi$.

Other important notions of dimension include the upper and lower packing dimension and upper and lower box-counting dimension, see \cite{Falconer, Mattila}.  The behaviour of these dimensions under projection is rather more subtle than for the Hausdorff dimension, but is nevertheless well-studied.  For example, for a compact planar set with (upper) packing dimension $s$, it is possible for the packing dimensions of all the projections to be strictly less than $\min\{1,s\}$.  This was first demonstrated by Maarit J\"arvenp\"a\"a \cite{Jarvenpaa} and general almost sure upper and lower bounds were given by Falconer and Howroyd \cite{FalconerHowroyd1}. Moreover, these bounds are sharp.   However, in the mid-1990s it became clear that each of these dimensions is at least \emph{almost surely constant} under projection, in the same sense that the Hausdorff dimension is.  The precise value of the constant is more complicated than simply $\min\{1, s\}$, but can be stated in terms of \emph{dimension profiles}, see \cite{FalconerHowroyd2, Howroyd}.

There are also natural higher dimensional analogues of these projection results. In particular, the higher dimensional analogue of Marstrand's theorem, where one considers projections of analytic sets $ F \subseteq \mathbb{R}^d$ to $k$-planes, was proved by Mattila in 1975 \cite{MattilaProjections}.  In this setting the almost sure value of the Hausdorff dimension is given by $\min\{k, \dim_\H F\}$ and `almost sure' refers to the natural invariant measure on the Grassmannian manifold $G_{d,k}$, which consists of all $k$ dimensional subspaces of $\mathbb{R}^d$.

For more information on the rich and fascinating topic of projections of fractal sets and measures, see the recent survey papers \cite{FalconerFraserJin, MattilaSurvey} and the references therein.

 The Assouad dimension is another notion of dimension, which has been very useful as a tool in several disparate areas of mathematics. In recent years it has been gaining more attention in the setting of fractal geometry and geometric measure theory.  As such, it is natural to consider the fundamental geometric properties of the Assouad dimension, such as its behaviour under orthogonal projections. We recall the definition here, but refer the reader to \cite{Robinson, Fraser, Luukkainen} for more details.  In particular, the Assouad dimension of a totally bounded set is always at least as big as the upper box-counting dimension (which is itself always at least as big as each of the Hausdorff, lower box-counting and upper and lower packing dimensions). For any non-empty subset $E \subseteq \mathbb{R}^d$ and $r>0$, let $N_r (E)$ be the smallest number of open sets with diameter less than or equal to $r$ required to cover $E$.  The \emph{Assouad dimension} of a non-empty set $F \subseteq \mathbb{R}^d$ is then given by
\begin{eqnarray*}
\dim_\text{A} F & = &  \inf \Bigg\{ \  s \geq 0 \  : \ \text{      $ (\exists \, C>0)$ $(\forall \, R>0)$ $(\forall \, r \in (0,R) )$ $(\forall \, x \in F)$ } \\ 
&\,& \hspace{40mm} \text{ $N_r\big( B(x,R) \cap F \big) \ \leq \ C \bigg(\frac{R}{r}\bigg)^s$ } \Bigg\}
\end{eqnarray*}
where $B(x,R)$ denotes the open Euclidean ball centred at $x$ with radius $R$. It is well-known that the Assouad dimension is always an upper bound for the Hausdorff dimension.

In this paper we prove that, unlike the Hausdorff, packing and box dimensions discussed above, the Assouad dimension of orthogonal projections of a compact set $F \subset \mathbb{R}^{2}$ need \emph{not} be almost surely constant as a function of the projection angle. However, in analogy with the result for Hausdorff dimension, the essential infimum of the said function is at least $\min\{\dim_{\textup{A}} F,1\}$; even if $F$ is not compact, or even not analytic. These results are discussed in Section \ref{mainGeneral}.

We establish the non-constancy result via a detailed study of the Assouad dimensions of projections of planar self-similar sets.  The Hausdorff dimension of projections of self-similar sets has attracted a lot of attention in recent years, see Section \ref{SSsection}, and thus it is natural to consider the analogous questions for Assouad dimension.  Our results for self-similar sets are discussed in detail in Section \ref{mainresults}, and the applications concerning non-constancy and Falconer's Theorem (mentioned in the abstract) will be presented in Sections \ref{nomarstrand} and \ref{nofalconer} respectively.  On route to proving our main result for self-similar sets, we obtain new information about the Assouad dimension of graph-directed self-similar sets in the line with overlaps, Theorem \ref{GDassouad}, which is the natural extension of \cite[Theorem 1.3]{Fraseretal} to the graph-directed setting and is of independent interest.  

Many questions remain unanswered by the results in this paper, and we pose some of them in Sections \ref{mainGeneral} through \ref{selfAffineSets}. In particular, very little is known for self-affine sets, and in higher dimensions.

\subsection{A short introduction to self-similar sets} \label{SSsection}

Self-similar sets are arguably the most fundamental class of fractal set and have been studied extensively, see \cite{hutchinson, Falconer}.  Let $\{S_i\}_{i \in \mathcal{I}}$ be a finite collection of contracting similarities mapping $[0,1]^d$ into itself.  By \emph{similarity}, we mean that for each $i \in \mathcal{I}$, there exists a similarity ratio $c_i \in (0,1)$ such that for all $x,y \in \mathbb{R}^d$ we have
\[
\lvert S_i(x) - S_i(y) \rvert \ = \ c_i \lvert x-y \rvert
\]
which means that the contractions $S_i$ scale uniformly by $c_i$ in every direction.  As such we may decompose each $S_i$ uniquely as
\[
S_i(x) = c_i O_i (x) + t_i  \qquad (x \in \mathbb{R}^d)
\]
where $O_i \in O(d)$ is a $d \times d$ orthogonal matrix and $t_i \in \mathbb{R}^d$ is a translation.  Here and later $O(d)$ denotes the orthogonal group consisting of all $d \times d$ orthogonal matrices and $SO(d)$ denotes the special orthogonal group, i.e., the subgroup of $O(d)$ consisting of orientation preserving matrices.  A fundamental result of Hutchinson \cite{hutchinson} states that there is a unique attractor of the iterated function system (IFS) $\{S_i\}_{i \in \mathcal{I}}$, that is, a unique non-empty compact set $F \subseteq [0,1]^d$ satisfying
\[
F \ = \ \bigcup_{i \in \mathcal{I}} S_i(F).
\]
The set $F$ is called self-similar and often has a rich fractal structure.  A self-similar set satisfies the open set condition (OSC) (for a given IFS defining it) if there exists a non-empty open set $U \subseteq [0,1]^d$ such that
\[
\bigcup_{i \in \mathcal{I}} S_i(U) \ \subseteq  \ U
\]
with the sets $S_i(U)$ pairwise disjoint. If $F$ satisfies the OSC, then its Hausdorff and Assouad dimensions are given by $\min\{s,d\}$ where $s$ is the \emph{similarity dimension} given by solving the equation
\[
\sum_{i \in \mathcal{I}} c_i^s = 1,
\]
often referred to as the \emph{Hutchinson-Moran formula}.  If the OSC is not satisfied, then the dimensions are more difficult to get hold of and indeed the Assouad and Hausdorff dimensions may be distinct \cite{Fraser, Fraseretal}.  One still expects the Hausdorff dimension to be given by $\min\{s,d\}$, unless there is a good reason for it not to be, such as exact overlaps in the construction, see \cite[Question 2.6]{update}.  Recent major advances were made in this area by Hochman \cite{hochman1, hochman2}.   There has also been intense interest in the dimension theory of the projections of self-similar sets in recent years, see the survey \cite{shmerkinsurvey}.  In particular, we wish to mention the following theorem:
\begin{thm} \label{hausdorffprojection}
Let $F\subseteq [0,1]^d$ be a self-similar set containing at least two points, suppose that the group generated by $\{O_i\}_{i \in \mathcal{I}}$ is dense in $O(d)$ or $SO(d)$ and fix $k \in \mathbb{N}$ less than $d$.  Then 
\[
\dim_\text{\emph{H}}  \pi F \ = \ \min\{k,\dim_\text{\emph{H}} F\}
\]
for \emph{\textbf{all}} $\pi \in G_{d,k}$. Here $G_{d,k}$ refers to the family of all orthogonal projections onto $k$-dimensional subspaces of $\mathbb{R}^{d}$. Also being `dense' in $O(d)$ or $SO(d)$ refers to the topology of pointwise convergence. 
\end{thm}

The particular interest of this result is that `dense rotations' guarantees that there are \emph{no} exceptional directions; a much stronger statement than Marstrand's Theorem which says that the exceptional directions form a null set. This result is essentially due to Hochman and Shmerkin's breakthrough work \cite{HochmanShmerkin}, although they stated the result assuming some separation conditions.  These conditions were explicitly removed by Farkas \cite{Farkas} and Falconer-Jin \cite{FalconerJin}.  The result in the planar case was obtained earlier by Peres and Shmerkin \cite{peresshmerkin}.

\section{Results}

\subsection{Projections of general sets} \label{mainGeneral} 

We parameterise orthogonal projections onto lines in $\mathbb{R}^{2}$ by  $\theta \in [0,\pi)$ in the natural way, by letting $\pi_\theta$ be the projection onto the  line $l_\theta$ passing through the origin and forming an angle $\theta$ with the positive $x$-axis. With this notation, our main result for general sets is the following:

\begin{thm}\label{mainTheoremForGeneralSets} Assume that $F \subset \mathbb{R}^{2}$. Then, for almost all $\theta \in [0,\pi)$,
\begin{displaymath} \dim_{\textup{A}} \pi_{\theta} F \geq \min\{\dim_{\textup{A}} F, 1\}. \end{displaymath}
\end{thm} 

Theorem \ref{mainTheoremForGeneralSets} will be proved in Section \ref{proofForGeneralSets}. The lower bound is obviously sharp, and in Section \ref{nomarstrand} we demonstrate by example that the inequality cannot be replaced by an equality: the function $\theta \mapsto \dim_{\textup{A}} \pi_{\theta} F$ need not be almost surely constant, and it can attain both the values $\dim_{\textup{A}} F$ and $1$ for a set of $\theta$'s with positive measure. We do not know if other values are possible here, or if there can be three distinct values:
\begin{ques}
Given a set $F \subset \mathbb{R}^{2}$, how many distinct values can the Assouad dimension of $\pi_\theta F$ assume for a set of $\theta$'s with positive measure? If there are only two such values, are they always $\dim_{\textup{A}} F$ and $1$? 
\end{ques}

\subsection{Projections of self-similar sets} \label{mainresults}

In this section we state our main results for self-similar sets, which are rather complete in the context of Assouad dimension. In the $2$-dimensional setting the groups $O(2)$ and $SO(2)$ are particularly simple: $O(2)$ consists of counterclockwise rotations by angles $\alpha \in [0,2\pi)$ and the corresponding reflections with orientation reversed and $SO(2)$ just consists of the rotations.  As such, the group generated by $\{O_i\}_{i \in \mathcal{I}}$ will be dense if and only if one of the $O_i$ rotates by an irrational multiple of $\pi$ and otherwise it will be discrete (in fact finite).

Here is the main result for self-similar sets:

\begin{thm} \label{assouadprojection}
Let $F\subseteq \mathbb{R}^{2}$ be self-similar set containing at least two points, and first suppose that the group generated by $\{O_i\}_{i \in \mathcal{I}}$ is discrete.  Then, for a given $\theta \in [0,\pi)$, we have:
\begin{enumerate}
\item If $\mathcal{H}^{\dim_\text{\emph{H}}  \pi_\theta F} ( \pi_\theta F) > 0$, then $\dim_\text{\emph{A}}  \pi_\theta F \ = \ \dim_\text{\emph{H}}  \pi_\theta F$
\item If $\mathcal{H}^{\dim_\text{\emph{H}}  \pi_\theta F} ( \pi_\theta F) =0$, then $\dim_\text{\emph{A}}  \pi_\theta F \ = \ 1$.
\end{enumerate}
Secondly, suppose that the group generated by $\{O_i\}_{i \in \mathcal{I}}$ is dense in $O(2)$ or $SO(2)$.  Then
\[
\dim_\text{\emph{A}}  \pi_\theta F \ = \ 1
\]
for \textbf{all} $\theta \in [0,\pi)$.
\end{thm}

The discrete rotations case of Theorem \ref{assouadprojection} will be proved in Section \ref{discreterotationsproof} and the dense rotations case will be proved in Section \ref{denserotationsproof}.

We emphasise that we assume no separation conditions for $F$, in particular the OSC may fail and the Hausdorff dimension of $F$ may be strictly smaller than the similarity dimension.  By Marstrand's Theorem, we have that, for almost all $\theta \in [0, \pi)$, $\dim_\text{H}  \pi_\theta F \ = \  \min\{1,\dim_\text{H}  F\}$, but Farkas \cite[Theorem 1.2]{Farkas} showed that in the discrete rotations case there is always at least one direction $\theta \in [0, \pi)$ where the dimension drops, i.e. $\dim_\text{H}  \pi_\theta F \ < \  \min\{1,\dim_\text{H}  F\}$, provided the Hausdorff dimension of $F$ is given by the similarity dimension and this value is less than or equal to 1.  In the dense rotations case Ero\u{g}lu \cite{eroglu} and Farkas \cite[Theorem 1.5]{Farkas} proved that $\mathcal{H}^{s} ( \pi_\theta F) =0$ for all $\theta \in [0,\pi)$. So, the dichotomy seen in the `discrete rotations' part persists for dense rotations, but case \emph{1.} never occurs.

In the dense rotations case, the Assouad dimension of $\pi_\theta F$ is constant and independent of the dimension of $F$.  Neither of these phenomena are generally manifest in the discrete case, but one can always find \emph{at least one} direction $\theta$ for which the Assouad dimension of $\pi_\theta F$ attains the maximal value of 1, independent of the dimensions of $F$, provided $F$ is not contained in a line.

\begin{thm} \label{assouadprojectionexamples}
Let $F\subseteq [0,1]^2$ be a self-similar set, which is not contained in a line.  Then there exists $\theta \in [0,\pi)$ such that
\[
\dim_\text{\emph{A}}  \pi_\theta F \ = \ 1.
\]
\end{thm}

We will prove Theorem \ref{assouadprojectionexamples} in Section \ref{assouadprojectionexamplesproof}. 




\subsection{There is no direct counterpart to Marstrand's Projection Theorem for Assouad dimension} \label{nomarstrand}

Our key application of Theorem \ref{assouadprojection} is that, unlike the Hausdorff, upper and lower box, and packing dimensions, the Assouad dimensions of orthogonal projections of a (compact) set are not almost surely constant in general.  Thus, we do not have a direct counterpart to Marstrand's Projection Theorem for Assouad dimension.

\begin{thm} \label{mainex}
For any $s$ satisfying $\log_5 3 < s <  1$, there exists a compact set $F \subseteq \mathbb{R}^2$ with Hausdorff and Assouad dimension equal to $s$ for which there are two non-empty disjoint intervals $I, J \subseteq [0, \pi)$ such that
\begin{align*}
&\dim_\text{\emph{A}} \pi_\theta F = s, \qquad \text{for all $\theta \in I$}\\
&\dim_\text{\emph{A}} \pi_\theta F = 1, \qquad \text{for almost all $\theta \in J$}.
\end{align*}
In particular, $\theta \mapsto \dim_\text{\emph{A}} \pi_\theta F$ is not an almost surely constant function.
\end{thm}

The restriction to $s > \log_5 3 \approx 0.6826$ in Theorem \ref{mainex} may well be an artefact of our method, but at present we are unaware how to construct lower-dimensional examples:
\begin{ques}
Given any $s \in [0,1)$, in particular $s \leq \log_{5} 3$, can one construct a compact planar set with Hausdorff dimension $s$, for which the Assouad dimension of the projections is not almost surely constant?
\end{ques}

Also, we do not know if the words `almost all $\theta \in J$' could be strengthened to `all $\theta \in J$' with a different construction:
\begin{ques}
Can one construct a compact planar set for which the Assouad dimension of the projection takes different values on two sets with non-empty interior?
\end{ques}

Note that Theorem \ref{mainex} demonstrates the apparently strange property that Assouad dimension can increase under projection (a Lipschitz map), which cannot happen for the dimensions discussed in Section 1.  This peculiarity of the Assouad dimension was observed previously in \cite[Section 3.1]{Fraser}.

The rest of this section will be dedicated to constructing an example with the properties required by Theorem \ref{mainex}. The set $F = F_c$ will actually be very simple: it will be a self-similar modification of the Sierpi\'nski triangle, where the contraction ratios are equal to $c \in (1/5,1/3)$, see Figure 1.  Fix $c \in (1/5,1/3)$ and let $F_c$ be the self-similar attractor of the IFS on $[0,1]^2$ given by $\{x \mapsto cx,\  x \mapsto cx+(0,1-c), \ x \mapsto cx+(1-c, 0)\}$.  Observe that  $F_c$ satisfies the \emph{open set condition} (OSC) and so $\dim_\text{H} F_c = \dim_\text{A} F_c =-\log 3/\log c  =: s \in (\log_5 3 , 1)$.  It follows from Theorem \ref{assouadprojection} that
\begin{enumerate}
\item  $\dim_\text{A} \pi_\theta F_c  =\dim_\text{H} \pi_\theta F_c  $  if and only if  $\mathcal{H}^{\dim_\text{H} \pi_\theta F_c}(\dim_\text{H} \pi_\theta F_c) >0 $
\item  $\dim_\text{A} \pi_\theta F_c  = 1 $ if and only if $\mathcal{H}^{\dim_\text{H} \pi_\theta F_c}(\dim_\text{H} \pi_\theta F_c) = 0 $.
\end{enumerate}
 In light of this dichotomy, and the fact that $\dim_\text{H} \pi_\theta F_c = s<1$ almost surely by Marstrand's Theorem, in order to complete the proof of Theorem \ref{mainex} it is sufficient to show that there is a non-empty open interval of $\theta$'s for which $ \mathcal{H}^s(\pi_\theta F_c) >0$, and a non-empty open interval of $\theta$'s within which $ \mathcal{H}^s(\pi_\theta F_c) =0$ almost surely.  The first of these tasks is straightforward, because one can easily find an open interval of $\theta$'s for which $\pi_\theta F_c$ is a self-similar set satisfying the open set condition and it then follows from standard results that $\dim_\H \pi_\theta F_c = s$ and $ \mathcal{H}^s(\pi_\theta F_c) >0$, see \cite[Chapter 9]{Falconer}.  The existence of such an interval relies on the assumption $c< 1/3$ since (after rescaling) the problem reduces to positioning three pairwise disjoint intervals of length $c$ inside the unit interval.

The second task is more delicate, but fortunately has already been solved by Peres, Simon and Solomyak \cite{PeresSimonSolomyak}.  They defined the set of \emph{intersection parameters} $\mathcal{IP} = \{ \theta : \pi_\theta \text{ is not injective on } F_c \}$ and proved that $\mathcal{IP}$ contains a non-empty interval provided $c \in (1/5,1/3)$ and, moreover, for almost every $\theta \in \mathcal{IP}$ we have $\mathcal{H}^s(\pi_\theta F_c) = 0$.  This can be found in \cite[Theorem 1.2(i) and Example 2.8]{PeresSimonSolomyak}.  We could also have used the 4-corner Cantor set with contraction parameter in the interval $(1/6,1/4)$, which was discussed in \cite{PeresSimonSolomyak}, but chose the Sierpi\'nski triangle because it yielded the least restrictive conditions on the dimension of $F=F_c$.

\begin{figure}[H] 
	\centering
	\includegraphics[width=110mm]{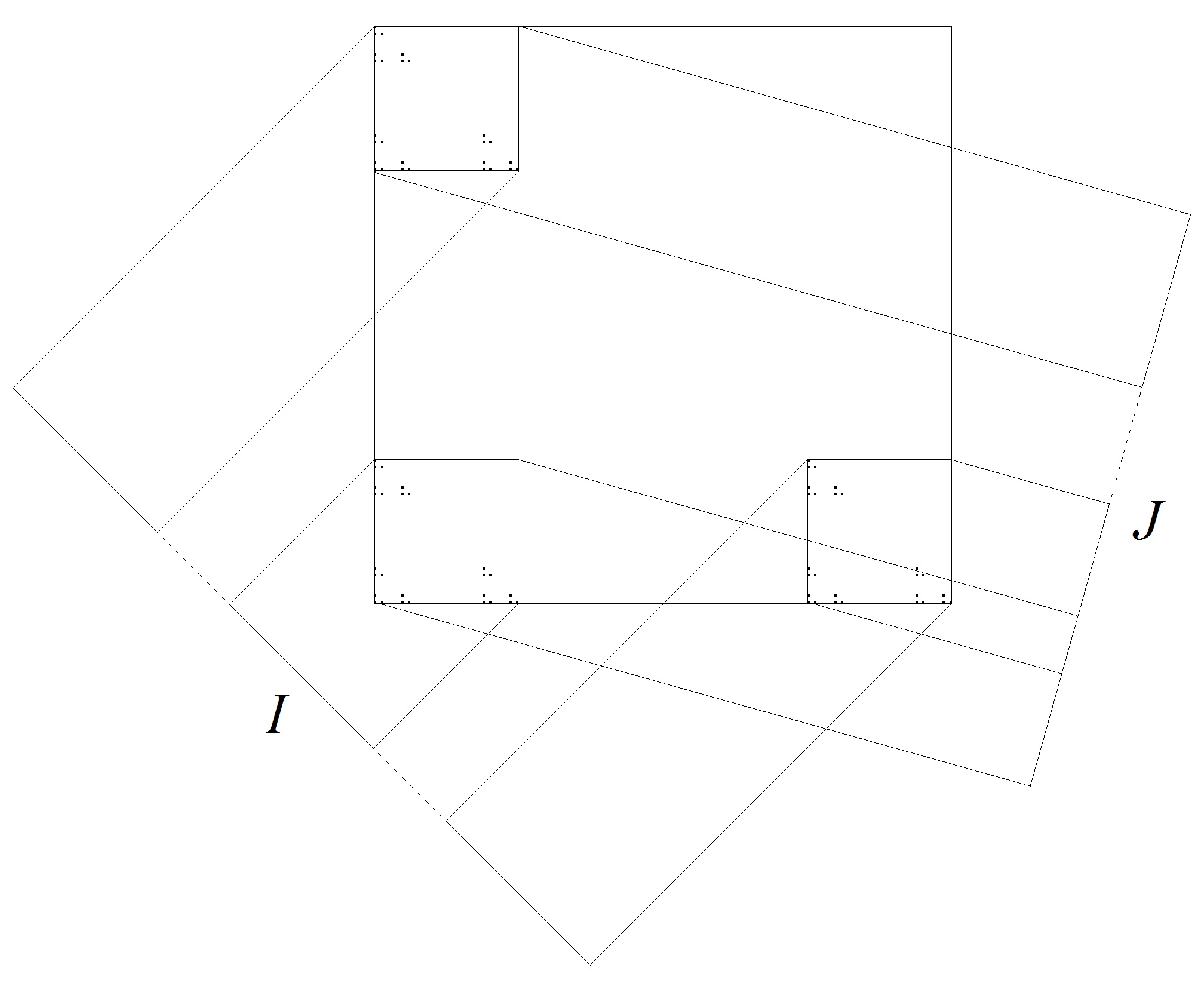}
\caption{The set $F_c$ (with $c = 1/4$) and two typical projections.  There is a small interval of projections $I$ for which the 3 pieces of $F_c$ project into pairwise disjoint intervals, meaning that projections in these directions satisfy the open set condition. There is also an interval $J$ within which the projection map is not injective. }
\end{figure}

\subsection{There is no direct counterpart to Falconer's Theorem for Assouad dimension} \label{nofalconer}

In 1988 Falconer proved a seminal result in the dimension theory of self-affine sets, see \cite{affine}.  Self-affine sets are closely related to self-similar sets, but the contractions in the defining IFS can be affine, i.e. $c_i O_i$ can be replaced with any contracting invertible $d \times d$ matrix.  This means that the scaled copies of the attractor can scale by different amounts in different directions, as well as being skewed or sheared, which makes them much more difficult to study.  Even if the OSC is satisfied, the dimensions may be hard to compute and the Assouad, box and Hausdorff dimensions may all be distinct.  Despite this, Falconer proved that the Hausdorff and box dimensions are generically equal to the \emph{affinity dimension}; the self-affine analogue of the similarity dimension.  Here `generically' means almost surely upon randomising the defining set of translations.  The affinity dimension depends only on the linear parts of the defining maps (as with the similarity dimension) and Falconer proved that for Lebesgue almost all choices of translation vectors $\{t_i\}_{i \in \mathcal{I}}$, the box and Hausdorff dimensions of the corresponding self-affine set are equal to the affinity dimension, provided the spectral norms of the matrices were all strictly less than 1/2.  In fact, Falconer's original proof required 1/3 here, but Solomyak relaxed this assumption to 1/2 and pointed out that this was optimal \cite{solomyak}.   Also the affinity dimension is always an upper bound for the Hausdorff and box dimensions, but this is not true for the Assouad dimension \cite{Mackay, Fraser}.  Specialising to the case of self-similar subsets of the line (which is a very restrictive class of self-affine sets), the assumption on norms is not required.  In particular, we have the following result due to Simon and Solomyak \cite{simonsolomyak}:

\begin{thm} \label{SSfalconer}
Fix a set $\{c_i\}_{i \in \mathcal{I}}$ with each $c_i \in (-1,1) \setminus \{0\}$ and let $s$ be the corresponding similarity (affinity) dimension given by
\[
\sum_{i \in \mathcal{I}} \lvert c_i \rvert^s = 1
\]
For a given set of translations $\textbf{t} = \{t_i\}_{i \in \mathcal{I}}$ with $t_i \in \mathbb{R}$, let $F_\textbf{t}$ denote the self-similar attractor of the IFS $\{x \mapsto c_i x + t_i\}_{i \in \mathcal{I}}$.  Then for Lebesgue almost all $\textbf{t}\in \mathbb{R}^{\lvert \mathcal{I} \rvert}$, one has
\[
\dim_\text{\emph{H}} F_\textbf{t} =  \dim_\text{\emph{B}} F_\textbf{t} =  \min\{1,s\}.
\] 
\end{thm}

It is natural to ask if such a theorem exists for the Assouad dimension, for example, is the Assouad dimension of a self-affine or self-similar set almost surely constant upon randomising the translations in the above manner.  We point out that the answer to this question is no, once again due to the example of Peres, Simon and Solomyak \cite{PeresSimonSolomyak} discussed in Section \ref{nomarstrand}.

\begin{thm} \label{mainex2}
Fix $c \in (1/5,1/3)$ and let $s = -\log 3/\log c <1$.  For a given set of translations $\textbf{t} = (t_1, t_2, t_3) \in \mathbb{R}^3$ let $F_\textbf{t}$ denote the self-similar attractor of the IFS $\{x \mapsto c x + t_i\}_{i=1}^3$.  Then there exists two non-empty disjoint open sets $U, V \subseteq \mathbb{R}^3$ such that
\begin{align*}
&\dim_\text{\emph{A}} F_\textbf{t}  = s, \qquad \text{for all $\textbf{t} \in U$}\\
&\dim_\text{\emph{A}} F_\textbf{t}  = 1, \qquad \text{for almost all $\textbf{t} \in V$}.
\end{align*}
In particular, $\textbf{t} \mapsto \dim_\text{\emph{A}} F_\textbf{t}$ is not an almost surely constant function.
\end{thm}

Similar to Section \ref{nomarstrand}, the open set $U$ is easy to find: choose $\textbf{t}' \in \mathbb{R}^3$ such that the OSC is satisfied (which can be done since $c<1/3$).  Then observe that the OSC is still satisfied for all $\textbf{t}$ in some open neighbourhood of $\textbf{t}'$ in $\mathbb{R}^3$.

Again, the second set $V$ is more subtle, but dealt with by Simon and Solomyak \cite{simonsolomyak} after re-parameterisation. Observe that for all $\textbf{t} \in \mathbb{R}^3$, the associated self-similar set is equal to the self-similar set associated to $(0,\lambda,1)$ for some $\lambda \in [0,1]$, after appropriate translating and re-scaling.  More precisely, consider the open subset of $\mathbb{R}^3$ given by
\[
X = \{\textbf{t} \, = \,  (t_1,t_2,t_3) \in \mathbb{R}^3 \, : \, 0 < t_1 <t_2<t_3<1\}
\]
and the map $\Xi: X \to (0,1)$ defined by
\[
 \Xi (t_1,t_2,t_3) \  = \  \frac{t_2-t_1}{t_3-t_1}  \ \in \  (0,1)
  \]
 which is easily seen to be a continuous surjection satisfying
\begin{equation} \label{abscont}
  \mathcal{L}^3 \circ \Xi^{-1}  \ll    \mathcal{L}^1,
\end{equation}
where $\mathcal{L}^3$ is $3$-dimensional Lebesgue measure restricted to $X$ and $\mathcal{L}^1$ is $1$-dimensional Lebesgue measure restricted to $(0,1)$. The attractor $F_\textbf{t}$ is affinely equivalent to the attractor corresponding to $(0,\Xi(\textbf{t}),1)$. Here, the affine rescaling is translation by $-t_1$ followed by rescaling by $(t_3+c)^{-1}$.  Moreover, the sets $F_\lambda$ for $\lambda \in (0,1)$ are just a smooth bijective reparamerisation of the sets $\pi_\theta F$ for $\theta \in (0,\pi/2)$ where $F$ is the modification of the Sierpi\'nski triangle from the previous section.  It follows by the example of Peres, Simon and Solomyak \cite{simonsolomyak} that there exists a non-empty open interval $J \subseteq (0,1)$ such that for almost all $\lambda \in J$, $\mathcal{H}^s(F_\lambda) = 0$.  It follows that the set $V = \Xi^{-1}(J) \subset \mathbb{R}^3$ is open and, by (\ref{abscont}), for Lebesgue almost all $\textbf{t} \in V$ we have $\mathcal{H}^s(F_\textbf{t}) = 0$.  Moreover, Theorem \ref{SSfalconer} implies that for  Lebesgue almost all $\textbf{t} \in V$ we have $\dim_\H F_\textbf{t} = s$ and so by \cite[Theorems 1.3]{Fraseretal} and \cite[Corollary 3.2]{FarkasFraser}, for  Lebesgue almost all $\textbf{t} \in V$, we have $\dim_\text{A} F_\textbf{t} = 1$.  This final implication also follows from our Theorem \ref{GDassouad}, stated in Section \ref{GDsection}.

\subsection{Higher dimensions and self-affine sets?}\label{selfAffineSets}

The higher dimensional variants of our results remain mostly open. In the dense rotations case, it is quite simple to show that the Assouad dimension is constant, regardless of the ambient dimension:
\begin{thm} \label{assouadconstant}
Let $F\subseteq [0,1]^d$ be self-similar,  suppose that the group generated by $\{O_i\}_{i \in \mathcal{I}}$  is dense in $O(d)$ or $SO(d)$ and fix $k \in \mathbb{N}$ less than $d$.  Then $\dim_\text{\emph{A}}  \pi F$ takes the same value for \textbf{all} $\pi \in G_{d,k}$. Recall that $G_{d,k}$ is the family of all orthogonal projections onto $k$-dimensional subspaces of $\mathbb{R}^{d}$.
\end{thm}  

We will prove Theorem \ref{assouadconstant} in Section \ref{assouadconstantproof}. Viewing the proof of the planar case in Section \ref{denserotationsproof}, it seems likely that this constant is always as large as possible: 

\begin{ques}
Let $F\subseteq [0,1]^d$ be a self-similar set containing at least two points,  suppose that the group generated by $\{O_i\}_{i \in \mathcal{I}}$  is dense in $O(d)$ or $SO(d)$ and fix $k \in \mathbb{N}$ less than $d$.  Then is it true that $\dim_\text{\emph{A}}  \pi F = k$  for \textbf{all} $\pi \in G_{d,k}$?
\end{ques}

The `non-dense case' is more complicated in higher dimensions, as `non-dense' no longer implies `discrete'.  

\begin{ques}
Let $F\subseteq [0,1]^d$ be self-similar and fix $k \in \mathbb{N}$ less than $d$.  Then is it true that $\dim_\text{\emph{A}}  \pi F$ is almost surely equal to either $k$ or $\dim_\text{\emph{H}}  \pi F$?
\end{ques}

In the `discrete' case, in all ambient dimensions, the projections are still graph-directed self-similar sets, but in higher dimensions the Assouad dimension of such sets is more complicated and can take values other than $k$ or $\dim_\text{H}  \pi F$ for particular directions, see \cite[Section 4.1]{Fraseretal}.

The examples in Section \ref{nofalconer} show that there is no `Falconer's Theorem for Assouad dimension'.  In light of \cite[Theorem 1.3]{Fraseretal} we know that for self-similar sets in the line with random translations, the situation is still relatively simple.  In particular, there are at most two values which $\dim_\text{A} F_{\textbf{t}}$ can take for a set of $\textbf{t}$ with positive measure: the similarity dimension, or 1.  The situation in higher dimensions and for self-\emph{affine} sets is still unclear however.

\begin{ques}
For $2 \leq d \in \mathbb{N}$ and a finite set of non-singular contracting $d \times d$ matrices $\{A_i\}_{i \in \mathcal{I}}$, how many values can the Assouad dimension of the attractor of the IFS $\{A_i+t_i\}_{i \in \mathcal{I}}$ take, for a set of translates $\textbf{t} = \{t_i\}_{i \in \mathcal{I}} \in \mathbb{R}^{d \lvert \mathcal{I} \rvert} $ with positive measure?
\end{ques}

\section{Proof of Theorem \ref{mainTheoremForGeneralSets}: projections of general sets}\label{proofForGeneralSets}
If $A,B > 0$, we will use the notation $A \lesssim_{p} B$ to signify that there exists a constant $C \geq 1$ depending only on $p$ such that $A \leq CB$. If the constant $C$ is absolute, we write $A \lesssim B$. The two-sided inequality $A \lesssim_{p} B \lesssim_{p} A$ is abbreviated to $A \sim_{p} B$. An example of this notation is given by $x^{4} + x^{2} \sim x^{4}$ for $x \in \mathbb{R}$. 

Let $\varepsilon > 0$ and $A \geq 1$ be parameters to be specified later (the choice of $\varepsilon$ will eventually be determined by a counter assumption, claiming that Theorem \ref{mainTheoremForGeneralSets} fails, and $A$ will depend on this $\varepsilon$; for the time being, $A$ and $\varepsilon$ are just some constants). We use the following notion of \emph{$(\delta,s)$-sets}:
\begin{defn}\label{deltaOne} Let $0 \leq s \leq d$. A finite set $P \subset B(0,1) \subset \mathbb{R}^{d}$ is called a $(\delta,s)$-set with parameters $A$ and $\varepsilon$, if the points in $P$ are $\delta$-separated (that is, $|p - q| \geq \delta$ for distinct $p,q \in P$), and 
\begin{equation} \label{deltaOneSet} |P \cap B(x,r)| \leq A\delta^{-\varepsilon}\left(\frac{r}{\delta}\right)^{s}, \qquad x \in \mathbb{R}^{d}, \: r \geq \delta. \end{equation}
Here, and throughout Section \ref{proofForGeneralSets}, the notation $|\cdot|$ stands for cardinality. 
\end{defn}

The following estimate with $\varepsilon = 0$ is Proposition 4.10 in \cite{Orponen}. Since the proof is verbatim the same in the case $\varepsilon > 0$, we do not repeat the details here. 
\begin{prop}\label{discreteMarstrand} Let $P \subset \mathbb{R}^{2}$ be a $(\delta,1)$-set with $m \in \mathbb{N}$ points, let $0 < \tau < 1$, and let $E \subset [0,\pi)$ be a $\delta$-separated collection of vectors such that
\begin{displaymath} N_{\delta}(\pi_{\theta}(P)) \leq \delta^{\tau}m, \qquad \theta \in E. \end{displaymath}
Then $|E| \lesssim A\delta^{\tau - 1 - \varepsilon}\log(1/\delta)$.
\end{prop}

The proposition below shows that if $0 \leq s \leq t \leq d$, then ``large" $(\delta,t)$-sets always contain ``large" $(\delta,s)$-sets. The argument below is practically repeated from \cite[Proposition A.1]{FasslerOrponen}, but we include it for the reader's convenience:

\begin{prop} Let $\delta > 0$ and $0 \leq s \leq t \leq d$. Assume that $P_{0} \subset B(0,1) \subset \mathbb{R}^{d}$ is a $\delta$-separated set with $|P_{0}| \geq c \delta^{-t}$ and satisfying $|P_{0} \cap B(x,r)| \leq C (r/\delta)^{t}$ for all $x \in \mathbb{R}^{d}$ and $r \geq \delta$, where $0 < c < C < \infty$ are constants. Then, there exists a subset $P \subset P_{0}$ with cardinality $|P| \gtrsim_{d} (c/C)\delta^{-s}$ satisfying $|P \cap B(x,r)| \lesssim_{d} (r/\delta)^{s}$ for all $x \in \mathbb{R}^{d}$ and $r \geq \delta$.
\end{prop}

\begin{proof} Without loss of generality, assume that $\delta = 2^{-k}$ for some $k \in \mathbb{N}$ and $P_0 \subset [0,1]^{d}$. Denote by $\mathcal{D}_{k}$ the dyadic cubes in $\mathbb{R}^{d}$ of side-length $2^{-k}$.  For a particular cube $Q$ we will write $d(Q)$ for its diameter and $\ell (Q)$ for the common side-length. First, find all the dyadic cubes in $\mathcal{D}_{k}$ which intersect $P_{0}$, and choose a single point of $P_{0}$ inside each of them. The finite set so obtained is denoted by $P_{1}$. Next, modify $P_{1}$ as follows. Consider the cubes in $\mathcal{D}_{k - 1}$. If one of these, say $Q^{k - 1}$, satisfies
\begin{displaymath} |P_{1} \cap Q^{k - 1}| > \left(\frac{d(Q^{k - 1})}{\delta}\right)^{s}, \end{displaymath}
remove points from $P_{1} \cap Q^{k - 1}$, until the reduced set $P_{1}'$ satisfies
\begin{displaymath} \frac{1}{2}\left(\frac{d(Q^{k - 1})}{\delta}\right)^{s} \leq |P_{1}' \cap Q^{k - 1}| \leq \left(\frac{d(Q^{k - 1})}{\delta}\right)^{s}. \end{displaymath}
Repeat this for all cubes in $\mathcal{D}_{k - 1}$ to obtain $P_{2}$. Then, repeat the procedure at all dyadic scales up from $\delta$, one scale at a time: whenever $P_{j}$ has been defined, and there is a cube $Q^{k - j} \in \mathcal{D}_{k - j}$ such that
\begin{displaymath} |P_{j} \cap Q^{k - j}| > \left(\frac{d(Q^{k - j})}{\delta}\right)^{s}, \end{displaymath}
remove points from $P_{j} \cap Q^{k - j - 1}$, until the reduced set $P_{j}'$ satisfies
\begin{equation}\label{appForm1} \frac{1}{2}\left(\frac{d(Q^{k - j})}{\delta}\right)^{s} \leq |P_{j}' \cap Q^{k - j}| \leq \left(\frac{d(Q^{k - j})}{\delta}\right)^{s}. \end{equation}
Stop the process when the remaining set of points, denoted by $P$, is entirely contained in some dyadic cube $Q_{0} \subset [0,1]^{d}$. Now, we claim that for every point $x \in P_{1}$ there exists a unique maximal dyadic cube $Q_{x} \subset Q_{0}$ such that $\ell(Q_{x}) \geq \delta$ and
\begin{equation}\label{appForm2} |P \cap Q_{x}| \geq \frac{1}{2}\left(\frac{d(Q_{x})}{\delta}\right)^{s}. \end{equation}
We only need to show that there exists \textbf{at least one} cube $Q_{x} \ni x$ satisfying \eqref{appForm2}; the rest follows automatically from the dyadic structure. If $x \in P$, we have \eqref{appForm2} for the dyadic cube $Q_{x} \in \mathcal{D}_{k}$ containing $x$. On the other hand, if $x \in P_{1} \setminus P$, the point $x$ was deleted from $P_{1}$ at some stage. Then, it makes sense to define $Q_{x}$ as the dyadic cube containing $x$, where the `last deletion of points' occurred. If this happened while defining $P_{j + 1}$, say, we have \eqref{appForm1} with $Q^{k - j} = Q_{x}$. But since this was the last cube containing $x$, where \textbf{any} deletion of points occurred, we see that that $P_{j}' \cap Q_{x} = P \cap Q_{x}$. This gives \eqref{appForm2}.
\newpage
Now, observe that the cubes $\{Q_{x} : x \in P_{1}\}$,
\begin{itemize}
\item cover $P_{0}$, because they cover every cube in $\mathcal{D}_{k}$ containing a point in $P_{0}$,
\item are disjoint, hence partition the set $P$.
\end{itemize}
These facts and \eqref{appForm2} yield the lower bound
\begin{eqnarray*} |P| \ = \ \sum |P \cap Q_{x}|  \ \gtrsim \  \delta^{-s}\sum d(Q_{x})^{s} &\geq& \delta^{-s} \sum d(Q_{x})^{t} \\ \\ 
&\gtrsim_{d}& \frac{\delta^{t - s}}{C} \sum |P_{0} \cap Q_{x}| \\ \\
&=& \frac{\delta^{t - s}}{C}|P_{0}| \\ \\
&\geq&  \frac{c\delta^{-s}}{C}.
\end{eqnarray*}
It remains to prove that $|P \cap B(x,r)| \lesssim (r/\delta)^{s}$ for all balls $B(x,r)$ with $r \geq \delta$. For dyadic cubes $Q \in \mathcal{D}_{l}$ with $l \leq k$ it follows immediately from the construction of $P$, in particular the right hand side of \eqref{appForm1}, that
\begin{displaymath} |P \cap Q| \leq \left(\frac{d(Q)}{\delta}\right)^{s}. \end{displaymath}
The statement for balls follows by observing that any intersection $P \cap B(x,r)$ can be covered by $\sim_{d} 1$ intersections $P \cap Q$, where $Q$ is a dyadic cube with $d(Q) \sim r$. 
\end{proof}

For later use, we record a corollary, stated in the terminology of $(\delta,s)$-sets:

\begin{cor}\label{frostman} Let $a > 0$ and $s \geq 1$. Assume that $P \subset B(0,1)$ is a $(\delta,s)$-set with parameters $A$ and $\varepsilon$, and cardinality $|P| \geq a\delta^{\varepsilon-s}$. Then, there exists a $(\delta,1)$-set $P' \subset P$ with parameters $A' \sim 1$ and $\varepsilon' = 0$, and with $|P'| \gtrsim (a/A)\delta^{2\varepsilon-1}$. 
\end{cor}

\begin{proof} Apply the previous proposition with $C = A\delta^{-\varepsilon}$ and $c = a\delta^{\varepsilon}$. \end{proof}

The corollary obviously fails for $s \leq 1$, but in this case the substitute will be the trivial observation that every $(\delta,s)$-set is automatically a $(\delta,1)$-set (with the same parameters $A$ and $\varepsilon$).  We are ready to prove Theorem \ref{mainTheoremForGeneralSets}.

\begin{proof}[Proof of Theorem \ref{mainTheoremForGeneralSets}] Write $\dim_{\textup{A}} F =: s \in [0,2]$, and let $\varepsilon > 0$ be the constant from Definition \ref{deltaOne} (to be specified shortly, after we have formulated a counter assumption). By the definition of Assouad dimension, two things hold:
\begin{itemize}
\item[(i)] We may find two sequences $(r_{i})_{i \in \mathbb{N}}$ and $(R_{i})_{i \in \mathbb{N}}$ of positive reals such that $0 < r_{i} < R_{i} < 1$, and $r_{i}/R_{i} \to 0$, and $N_{r_{i}}(B(x_{i},R_{i}) \cap F) \geq (R_{i}/r_{i})^{s - \varepsilon}$ for some $x_{i} \in \mathbb{R}^{2}$.
\item[(ii)] For any $0 < r \leq R < 1$ and $x \in \mathbb{R}^2$, we have $N_{r}(B(x,R) \cap F) \leq A(R/r)^{s + \varepsilon}$ for some constant $A = A_{\varepsilon,F} \geq 1$. We now declare that we use this constant $A$ in Definition \ref{deltaOne}. 
\end{itemize}

Fix $0 < r_{i} < R_{i} < 1$ as in (i), write $\delta_{i} := r_{i}/R_{i}$, and let $P_{i}' \subset B(x_{i},R_{i}) \cap F$ be an $r_{i}$-separated set with $|P_{i}'| \geq \delta_{i}^{-s + \varepsilon}$. Then, if $T_{x_{i},R_{i}}$ is the simplest possible affine mapping taking $B(x_{i},R_{i})$ to $B(0,1)$, we find that $P_{i} := T_{x_{i},R_{i}}(P_{i}')$ is a $\delta_{i}$-separated set with cardinality $\geq \delta_{i}^{-s + \varepsilon}$. Furthermore, $P_i$ is a $(\delta_{i},s)$-set with parameters $A$ and $\varepsilon$: if $r \geq \delta_{i}$, then $rR_{i} \geq r_{i}$, and so (ii) gives
\begin{eqnarray*} 
|P_{i} \cap B(y,r)| \ = \  |P_{i}' \cap T_{x_{i},R_{i}}^{-1}(B(y,r))| & \leq& A\left(\frac{rR_{i}}{r_{i}} \right)^{s + \varepsilon} \\ \\
& =& A\left(\frac{r}{\delta_{i}} \right)^{s + \varepsilon} \\ \\
&\leq& A\delta_{i}^{-\varepsilon}\left(\frac{r}{\delta_{i}} \right)^{s}. 
\end{eqnarray*}

If $s > 1$, we use Corollary \ref{frostman} to find a $(\delta_{i},1)$-set $\tilde{P}_{i} \subset P_{i}$ such that $|\tilde{P}_{i}| \gtrsim \delta_{i}^{2\varepsilon-1}/A$. If $s \leq 1$, we just repeat the observation that $P_{i}$ is a $(\delta_{i},1)$-set. These alternatives will lead to the generic lower bounds $\dim_{\textup{A}} \pi_{\theta}F \geq s$ or $\dim_{\textup{A}} \pi_{\theta}F \geq 1$, respectively: their proofs are so similar that we only record explicitly the case $s \leq 1$. 

To reach a contradiction, we assume that Theorem \ref{mainTheoremForGeneralSets} fails for the particular set $F$ we are considering: thus, we assume that there is a constant $\tau > 0$, and a positive measure set of directions $E_{0} \subset [0,\pi)$ such that $\dim_{\textup{A}} \pi_{\theta}F < s - \tau$ for all $\theta \in E_{0}$. Then, we finally fix $\varepsilon := \tau/10$. In particular, we have $\delta^{\tau - 2\varepsilon}\log(1/\delta) \lesssim_{\tau} \delta^{\tau/2}$ for $0 < \delta < 1$. For each direction $\theta \in E_{0}$, one should be able to find infinitely many values of $i \in \mathbb{N}$ such that 
\begin{displaymath} N_{r_{i}}(B(x,R_{i}) \cap \pi_{\theta}F) < \left(\frac{R_{i}}{r_{i}} \right)^{s - \tau} \end{displaymath}
for all $x \in \mathbb{R}$; otherwise clearly $\dim_{\textup{A}} \pi_{\theta}F \geq s - \tau$. 

With the (easier) Borel-Cantelli lemma in mind, we define the sets $E_{0}^{i}$, $i \in \mathbb{N}$, by
\begin{displaymath} E_{0}^{i} := \left\{\theta \in E_{0} : N_{r_{i}}(B(t,R_{i}) \cap \pi_{\theta}F) < \left(\frac{R_{i}}{r_{i}} \right)^{s - \tau} \text{ for all } t \in \mathbb{R}\right\}. \end{displaymath}
According to the preceding discussion, every point of $E_{0}$ should lie in $E_{0}^{i}$ for infinitely many values of $i \in \mathbb{N}$. So, by the Borel-Cantelli lemma, we wish to show that
\begin{equation}\label{summability} \sum_{i \in \mathbb{N}} \mathcal{H}^{1}(E_{0}^{i}) < \infty. \end{equation}
This forces $\mathcal{H}^{1}(E_{0}) = 0$ and brings the desired contradiction.

To estimate the measure of $E_{0}^{i}$, it suffices to estimate the maximum number of $\delta_{i}$-separated points in $E_{0}^{i}$, where $\delta_{i} = r_{i}/R_{i}$: if this number is $N_{i}$, we have $\mathcal{H}^{1}(E_{0}^{i}) \lesssim \delta_{i}N_{i}$. Since $P_{i}$ is a $(\delta_{i},1)$-set with parameters $A$ and $\varepsilon$, and cardinality $|P_{i}| \geq \delta^{-s + \varepsilon}$, Proposition \ref{discreteMarstrand} and the choice of $\varepsilon$ imply that there are $\lesssim A\delta_{i}^{\tau - 1 - 2\varepsilon}\log(1/\delta_{i}) \lesssim_{\tau} A\delta_{i}^{\tau/2 - 1}$ angles $\theta \in [0,\pi)$, which are $\delta_{i}$-separated and such that $N_{\delta_{i}}(\pi_{\theta}(P_{i})) \leq \delta_{i}^{\tau - s}$. 

We now argue that if $\theta$ satisfies the converse inequality, then $\theta \notin E_{0}^{i}$. Indeed, if $N_{\delta_{i}}(\pi_{\theta}(P_{i})) > \delta_{i}^{\tau - s}$, then, for a certain $t \in \mathbb{R}$ (depending on the mapping $T_{x_{i},R_{i}}$) we have 
\begin{displaymath} N_{r_{i}}(B(t,R_{i}) \cap \pi_{\theta}F) \geq N_{r_{i}}(B(t,R_{i}) \cap \pi_{\theta}(T_{x_{i},R_{i}}^{-1}(P_{i}))) > \delta_{i}^{\tau - s} = \left(\frac{R_{i}}{r_{i}}\right)^{\tau - s}, \end{displaymath}
and so $\theta \notin E_{0}^{i}$. These observations show that $N_{i} \lesssim_{\tau} A\delta_{i}^{\tau/2 - 1}$, and hence $\mathcal{H}^{1}(E_{0}^{i}) \lesssim_{\tau} A\delta_{i}^{\tau/2}$. If $\delta_{i} \to 0$ fast enough (as we may assume), this proves \eqref{summability} and the theorem. \end{proof}

\section{Proof of Theorem \ref{assouadprojection}: the discrete rotations case} \label{discreterotationsproof}

The discrete rotations case of Theorem \ref{assouadprojection} will be proved by extending the work of \cite{Fraseretal} on self-similar sets in the line to graph-directed self-similar sets in the line, and then recalling that projections of planar self-similar sets with discrete rotations onto lines are precisely graph-directed self-similar sets, see Section \ref{GDequalsprojections}.

\subsection{Graph-directed self-similar sets in the line} \label{GDsection}

Graph-directed self-similar sets are an important and natural generalisation of self-similar sets.  First considered by Mauldin and Williams \cite{Mauldin}, roughly speaking one has a family of sets rather than a single set (as in the self-similar case) and each member of the family is made up of scaled copies of other sets in the family.  More precisely, let $\Gamma = (\mathcal{V}, \mathcal{E})$ be a finite connected directed graph, where $\mathcal{V}$ is a finite vertex set and $\mathcal{E}$ is a finite set of edges, each of which starts and ends at a vertex. Note that there may be multiple edges connecting a particular pair of vertices.   To each $e \in \mathcal{E}$, associate a contracting similarity map $S_e : \mathbb{R} \to \mathbb{R}$  with contraction ratio $c_e$. We assume for convenience that $S_e( [0,1]) \subseteq [0,1]$.  For $u,v \in \mathcal{V}$, let $\mathcal{E}_{u,v} \subseteq \mathcal{E}$ be the set of edges from $u$ to $v$.  Then there exists a unique $\lvert \mathcal{V} \rvert$-tuple of compact nonempty sets $\{F_v\}_{v \in \mathcal{V}}$,  each contained in $[0,1]$, satisfying
\[
F_v \ = \ \bigcup_{u \in \mathcal{V}} \bigcup_{e \in \mathcal{E}_{u,v}} S_e(F_u).
\]
The family $\{F_v\}_{v \in \mathcal{V}}$ is the \emph{family of graph-directed self-similar sets}.  Since the directed graph $\Gamma$ is connected, it follows that the sets $F_v$ have a common Hausdorff dimension.  Here we only consider graph-directed self-similar sets in the line, but one can consider more general models in the same way where, for example, one works with more general maps or in higher dimensions.

Zerner \cite{zerner}, following Lau-Ngai \cite{laungai}, defined the \emph{weak separation property} (WSP) for self-similar sets.  This is weaker than the open set condition (OSC) but in many cases plays a similar role in that if the WSP is satisfied, then the overlaps in the construction are controllable and the attractor shares many properties with attractors in the OSC case.    Das and Edgar \cite{GDWSP} generalised the WSP to the graph-directed setting and we will use their condition here.  For $u,v \in \mathcal{V}$, let
\[
\mathcal{F}_{u,v} = \{ S_\textbf{e}^{-1} \circ S_\textbf{f} : \textbf{e}, \textbf{f} \in \mathcal{E}_{u,v}^*\}
\]
where $\mathcal{E}_{u,v}^*$ denotes the set of all finite directed paths in the graph going from $u$ to $v$ and we write $S_\textbf{e}$ for the similarity defined by traversing $\textbf{e}$ and composing the similarity maps corresponding to each edge in the appropriate order.  Also $c_\textbf{e}$ will denote the similarity ratio of $S_\textbf{e}$.
\begin{defn}[GDWSP]
The graph $\Gamma$ and associated mappings satisfy the \emph{graph-directed weak separation property (GDWSP)} if for any (or equivalently all) $v \in \mathcal{V}$, the identity is an isolated point of $\mathcal{F}_{v,v}$ in the topology of pointwise convergence. 
\end{defn}

We do not explicitly define the WSP here but note that it is precisely the GDWSP in the 1-vertex case, i.e., in the case where graph-directed self-similar sets reduce to self-similar sets.  Das and Edgar gave many equivalent formulations of the GDWSP \cite{GDWSP} which parallel Zerner's list of equivalent definitions of the WSP in the self-similar (or 1-vertex) case \cite[Theorem 1]{zerner}.  Das and Edgar then went on to show that many of the properties of self-similar sets satisfying the WSP generalise to graph-directed self-similar sets satisfying the GDWSP.  Our main result for graph-directed self-similar sets is the following.

\begin{thm} \label{GDassouad}
Let $\{F_v\}_{v \in \mathcal{V}}$ be a family of graph-directed self-similar sets in $[0,1]$ with common Hausdorff dimension $s<1$.  Then
\begin{enumerate}
\item If the GDWSP is satisfied, then for all $v \in \mathcal{V}$ we have $\mathcal{H}^{s} (F_v) > 0$ and $\dim_\text{\emph{A}}   F_v =  \dim_\text{\emph{H}}   F_v  =   s  <  1$.
\item If the GDWSP is not satisfied, then for all $v \in \mathcal{V}$ we have $\mathcal{H}^{s} (F_v) = 0$ and $\dim_\text{\emph{A}}   F_v  = 1$.
\end{enumerate}
\end{thm}

This result can be seen as a generalisation of \cite[Theorem 1.3]{Fraseretal} and \cite[Corollary 3.2]{FarkasFraser} to the graph-directed setting.  In particular, it gives a precise dichotomy for the Assouad dimension of graph-directed self-similar sets on the real line and proves that if the Hausdorff dimension is strictly less than 1, then the GDWSP is equivalent to, for example, positivity of the Hausdorff measure in the Hausdorff dimension.  The `Hausdorff  measure in the Hausdorff dimension' of a set $E$ is $\mathcal{H}^{\dim_\mathrm{H} E}(E)$.   This shows that the GDWSP can be viewed as a property of the \emph{sets} $F_v$, rather than the defining graph.

The proof of this theorem is divided into two parts, proving \emph{1.} (concerning weak separation) and \emph{2.} (concerning lack of weak separation) respectively.

\subsubsection{Proof of \emph{1.}: systems with weak separation}

This result is proved by combining previous work of Das and Edgar on the GDWSP \cite{GDWSP} with Falconer's implicit theorems \cite{quasi} and recent work of Farkas and Fraser on Ahlfors regularity and Hausdorff measure \cite{FarkasFraser}.

Let $s$ be the common Hausdorff dimension of the graph-directed family $\{F_v\}_{v \in \mathcal{V}}$ and fix $v \in \mathcal{V}$. Since the GDWSP is satisfied, it follows from \cite[(4b)]{GDWSP} that there exists a uniform constant $K>0$ such that for all sets $U \subseteq F_v$ and $u \in \mathcal{V}$,
\[
\# \left\{ S_\textbf{e} : \textbf{e} \in \mathcal{E}^*_{u,v}, \,  c_{\min} \lvert U\rvert <  c_\textbf{e} \leq \lvert U \rvert , \, S_\textbf{e}(F_u) \cap U \neq \emptyset \right\} \ < \ K.
\]
It then follows from \cite[Theorem 2]{quasi} that
\[
\mathcal{H}^s(F_v) \geq c_{\min}^s K^{-1} \lvert \mathcal{V} \rvert^{-1}> 0.
\]
We note that this observation was made explicitly in the self-similar case by Zerner \cite[Corollary on p. 3535]{zerner}.

Now that we have established positivity of the Hausdorff measure, the equality of the Hausdorff and Assouad dimensions is an immediate consequence of \cite[Corollary 3.1]{FarkasFraser}, which proved that for a graph-directed self-similar set $E$, the Hausdorff measure of $E$ is positive in the Hausdorff dimension if and only if $E$ is Ahlfors regular (that is, $\mathcal{H}^{\dim_\mathrm{H} E}(B(x,r) \cap E) \sim r^{\dim_\mathrm{H} E}$ for all $x \in E$ and $0 < r \leq \textup{diam}(E)$). We note that \cite[Corollary 3.1]{FarkasFraser} is stated for irreducible subshifts of finite type rather than graph-directed sets, but these two notions are equivalent: see the discussion following \cite[Corollary 3.1]{FarkasFraser} and \cite[Proposition 2.5--2.6]{FarkasFraser}.  Alternatively, see \cite[Propositions 2.2.6 and 2.3.9]{LMsym}.  It is also well-known and straightforward to prove that Ahlfors regular sets have equal Hausdorff and Assouad dimensions.

\subsubsection{Proof of \emph{2.}: systems without weak separation}

This part follows the proof of \cite[Theorem 3.1]{Fraseretal}.  Indeed, one just has to check that the proof there extends to the graph-directed setting.  For completeness we include the argument, but in a slightly streamlined form.

 Since the GDWSP is not satisfied, the identity, $I$, is not an isolated point of $\mathcal{F}_{v,v}$ in the topology of pointwise convergence, which is equivalent to the uniform operator topology in this setting since all the maps in $\mathcal{F}_{v,v}$ are similarities. This means we may find a sequence $(\textbf{e}_k, \textbf{f}_k)\in \mathcal{E}_{v,v}^* \times \mathcal{E}^*_{v,v}$ such that
\[
0< \| S_{\textbf{e}_k}^{-1} \circ  S_{\textbf{f}_k}- I \| \to 0
\]
as $k \to \infty$, where $\| \cdot \|$ denotes the operator norm.  Moreover, we may assume that for all $k$ the maps $S_{\textbf{e}_k}$ and $S_{\textbf{f}_k}$ have no reflectional components: if this was not the case, then there must be an edge $\textbf{g} \in \mathcal{E}_{v,v}$ which also contains a reflection and then, whenever $S_{\textbf{e}_k}$ and $S_{\textbf{f}_k}$ both contain reflections, one may replace them by $S_{\textbf{g}\textbf{e}_k}$ and $S_{\textbf{g}\textbf{f}_k}$ in the sequence.  Both of these maps have no reflectional component and 
\[
 S_{\textbf{g}\textbf{e}_k}^{-1} \circ  S_{\textbf{g}\textbf{f}_k}  = S_{\textbf{e}_k}^{-1} \circ  S_{\textbf{f}_k}
\]
and so the convergence is unaffected.

Let $\phi_k = S_{\textbf{e}_k}^{-1} \circ  S_{\textbf{f}_k} - I$, which is either a similarity or a non-zero constant function.   We may choose a point  $a \in F_v$, an edge $\textbf{e}' \in \mathcal{E}^*_{v,v}$, and a small radius $r>0$ such that $S_{\textbf{e}'}(a) = a$ and for all $k$, we have $\phi_k(B(a,r)) \subset (0,\infty)$.  This can be achieved by choosing two such fixed points $a, a'$ and $r< \lvert a-a'\rvert/2$ and then observing that for all $k$ either $\phi_k(B(a,r))$ or $\phi_k(B(a',r))$ does not include zero and so lies to the left or right of zero.  Then by choosing a subsequence (and flipping the axes if necessary) we can achieve this using $a$ or $a'$ for all $k$. Moreover, by an affine change of coordinates we may assume without loss of generality that $a=0$.  Finally we may assume that $S_{\textbf{e}'}$ does not contain a reflection, since if it did we could replace it by $S_{\textbf{e}'}^2$.  Denote the map $S_{\textbf{e}'}$ by $T$ and let $c = c_{\textbf{e}'} \in (0,1)$.

For $k \in \mathbb{N}$, let 
\[
\delta_k = \inf\{\phi_k(x) : x \in B(0,r/2)\}>0
\]
observing that $\delta_k \to 0$ as $k \to \infty$.    Choose $M \in \mathbb{N}$ large enough to ensure that
\[
T^M(F_v) \subseteq B(0,r/2)
\]
which we may do since $T$ is a contraction and fixes $0$. Observe that for any $m \in \mathbb{N}$
\[
(T^{-m} \circ S_{\textbf{e}_k}^{-1} \circ  S_{\textbf{f}_k} \circ T^m-I)(x) = c^{-m} \phi_k(T^m(x))
\]
and so for $m \geq M$ and any $x \in F_v$ we have
\begin{equation} \label{estimate}
c^{-m} \delta_k \leq (T^{-m} \circ S_{\textbf{e}_k}^{-1} \circ  S_{\textbf{f}_k} \circ T^m-I)(x)  \leq 3 c^{-m} \delta_k,
\end{equation}
observing that $\phi_k(T^m(x)) \leq 3\delta_k$ since $\phi_k$ is a similarity and $\inf\{\phi_k(x) : x \in B(0,r)\}>0$ by assumption. Note that the definition of $\delta_k$ used $r/2$ here instead of $r$.  Let $\varepsilon>0$ and choose $k_j, m_j \in \mathbb{N}$ and maps $g_j, h_j$ by induction on $j$ as follows.  Begin by choosing $k_1$ large enough that
\[
\delta_{k_1} < \varepsilon c^M
\]
and then choose $m_1 \geq M$ such that
\[
c^{-m_1} \delta_{k_1} < \varepsilon \leq c^{-m_1-1} \delta_{k_1}.
\]
Also define
\[
g_1 = T^{-m_1} \circ  S_{\textbf{e}_{k_1}}^{-1} \qquad \text{and} \qquad h_1 =  S_{\textbf{f}_{k_1}} \circ T^{m_1}.
\]
For $j \geq 2$, similar to above choose $k_j, m_j$ such that
\begin{equation} \label{bounds333}
c_{j-1} c^{-m_j} \delta_{k_j} < \varepsilon \leq c_{j-1} c^{-m_j-1}\delta_{k_j},
\end{equation}
where $c_{j-1}$ is the similarity ratio of $g_{j-1}$, and define 
\[
g_j = g_{j-1} \circ  T^{-m_j} \circ  S_{\textbf{e}_{k_j}}^{-1} \qquad \text{and} \qquad h_j = S_{\textbf{f}_{k_j}} \circ  T^{m_j} \circ  h_{j-1}. 
\]
It follows that for all $j$ and all $x \in F_v$
\begin{eqnarray*}
(g_j \circ h_j -g_{j-1} \circ h_{j-1})(x) &=&  g_{j-1} \circ  T^{-m_j} (  S_{\textbf{e}_{k_j}}^{-1}   \circ S_{\textbf{f}_{k_j}} \circ  T^{m_j} \circ  h_{j-1}(x))\\ 
&\,& \qquad - g_{j-1} \circ  T^{-m_j} (  T^{m_j} \circ h_{j-1}(x)) \\  
&=& c_{j-1} c^{-m_j} \phi_{k_j}(T^{m_j} \circ h_{j-1}(x))
\end{eqnarray*}
which, by (\ref{bounds333}), implies that
\begin{equation} \label{separation}
c \varepsilon \,  \leq  \,  c_{j-1} c^{-m_j}\delta_{k_j} \,  \leq \, (g_j \circ h_j -g_{j-1} \circ h_{j-1})(x) \, \leq \, 3 c_{j-1} c^{-m_j}\delta_{k_j}   \, \leq \,  3 \varepsilon.
\end{equation}
We will now prove by induction that for all $n \in \mathbb{N}$ we have
\begin{equation} \label{contains}
\{g_n^{-1}(0)\} \cup\{g_n^{-1} \circ g_j \circ h_j(0) : j =1, \dots, n\} \ \subseteq \ F_v.
\end{equation}
For $n=1$, this is evident.  Given the claim for $n-1$, choose $\textbf{f} \in \mathcal{E}^*_{v,v}$ such that $S_\textbf{f} = S_{\textbf{e}_{k_n}} \circ T^{m_n}$ and so, in particular, $ g_n \circ  S_\textbf{f} \circ g_{n-1}^{-1}$ is the identity.  Observe that
\begin{eqnarray*}
F_v &\supseteq&  g_n^{-1}\circ g_n \circ  S_\textbf{f}(F_v) \\
&\supseteq& \{g_n^{-1}\circ g_n \circ  S_\textbf{f} \circ g_{n-1}^{-1}(0)\} \cup \{ g_n^{-1} \circ g_n \circ S_\textbf{f}  \circ g_{n-1}^{-1} \circ g_j \circ h_j(0) : j =1, \dots, n-1\}  \\
&\,& \qquad \qquad  \qquad \qquad \qquad  \qquad \qquad \qquad  \qquad \qquad \qquad   \text{(by inductive hypothesis)} \\ 
 &=& \{g_n^{-1}(0)\} \cup\{g_n^{-1} \circ g_j \circ h_j(0) : j =1, \dots, n-1\}.
\end{eqnarray*}
Finally, the missing point, $g_n^{-1} \circ g_n \circ h_n(0) = h_n(0)$ is clearly in $F_v$ which completes the inductive argument.  For $\varepsilon = 1/n$, consider the ball centered at $x=g_n^{-1}(0) \in F_v$ with radius $R = 3 c_n^{-1}$, which is small for large $n$.  By (\ref{separation}) and (\ref{contains}) we have
\[
B(x,R) \cap F_v \ \supseteq \ \{g_n^{-1} \circ g_j \circ h_j(0) : j =1, \dots, n\}
\]
and the $n$ points on the right hand side are all separated by at least $ c_n^{-1}c/n$.  Setting $r=c_n^{-1}c/(2n)$, this means that
\[
N_r \big(B(x,R) \cap F_v ) \, \geq \,  n \, = \, (c/6)  \left(\frac{R}{r}\right)^1
\]
and this yields $\dim_\text{A} F_v = 1$ as required.

All that remains to complete the proof is to show that the Hausdorff measure of $F_v$ is zero in the Hausdorff dimension, but this follows again by \cite[Corollary 3.1]{FarkasFraser} since if the Hausdorff measure was positive, then $F_v$ would be Ahlfors regular and thus have Assouad dimension equal to $s<1$.

\subsection{Application to projections of planar self-similar sets} \label{GDequalsprojections}

The discrete rotations case in Theorem \ref{assouadprojection} follows immediately from Theorem \ref{GDassouad}.  The reason for this is that projections $\pi_\theta F$ are graph-directed self-similar subsets of $[0,1]$ (following appropriate rescaling and translating).  For example, if the group generated by $\{O_i\}_{i \in \mathcal{I}}$ is a discrete subgroup of $SO(2)$, then it is isomorphic to the finite cyclic group of order $n$ for some $n \in \mathbb{N}$.  It follows that, for a given $\theta \in [0,\pi)$,  the family
\[
\{\pi_{\theta+2\pi k/n (\text{mod } \pi)} F\}_{k=0}^{n-1}
\]
is a family of graph-directed self-similar sets with associated graph and IFS inherited from the IFS $\{S_i\}_{i \in \mathcal{I}}$ in the natural way. If the group contains orientation reversing maps, then the situation is not much more complicated. For the details, the reader is referred to \cite[Theorem 1.1]{Farkas}, which also handles the higher dimensional setting.

\section{Proof of Theorem \ref{assouadprojection}: the dense rotations case} \label{denserotationsproof}

In this section, it is be convenient to treat $\pi_{\theta}$ as a mapping $\mathbb{R}^{2} \to \mathbb{R}$ rather than $\mathbb{R}^{2} \to \operatorname{span}(\cos \theta,\sin \theta)$. In other words,
\begin{displaymath} \pi_{\theta}(x) := x \cdot (\cos \theta,\sin \theta) \end{displaymath}
for $\theta \in [0,\pi)$ and $x \in \mathbb{R}^{2}$.

We begin by reducing the proof of the dense rotations case of Theorem \ref{assouadprojection} to studying a very simple class of IFS.  We will write $\mathcal{I}^{*} = \bigcup_{k \in \mathbb{N}} \mathcal{I}^k$ for the set of finite words over $\mathcal{I}$ and for $\textbf{i} = (i_1, i_2, \dots, i_k) \in \mathcal{I}^{*}$, we write $S_\textbf{i} = S_{i_1} \circ  S_{i_2} \circ \cdots  \circ S_{i_k}$ and $c_\textbf{i} = c_{i_1} c_{i_2} \cdots   c_{i_k}$ for the contraction ratio of $S_\textbf{i}$.

\begin{lma} \label{reduction}
Let $F \subseteq [0,1]^2$ be a self-similar set such that the group generated by $\{O_i\}_{i \in \mathcal{I}}$ is dense in $O(2)$ or $SO(2)$.  Then $F$ contains a self-similar set $E$ which is the attractor of an IFS consisting of two maps $\{S_1, S_2\}$, both of which have the same contraction ratio $c \in (0,1)$, the same orthogonal component $O \in SO(2)$ corresponding to anti-clockwise rotation by an irrational multiple of $\pi$, and such that $S_1(E) \cap S_2(E) = \emptyset$.
\end{lma}

\begin{proof}
This lemma is almost trivial since we are not concerned with losing any dimension.  Since we are in the dense rotations case, we may choose finite words $\textbf{i}_1, \textbf{i}_2 \in \mathcal{I}^*$ such that the orthogonal parts $O_{\textbf{i}_1}, O_{\textbf{i}_2} \in SO(2)$ are orientation preserving and correspond to anti-clockwise rotations by angles $ 2\pi\alpha,  2\pi\beta \in [0, 2\pi)$ respectively, where $\alpha, \beta \in [0,1)$ and at least one of which is irrational.  Moreover, since $F$ is assumed not to be a single point we can guarantee that $S_{\textbf{i}_1}(F) \cap S_{\textbf{i}_2}(F) = \emptyset$.  Assume without loss of generality that $\alpha$ is irrational and let $1 \leq m,n \in \mathbb{N}$ be such that $m \alpha + n \beta$ is irrational. This can be done irrespective of $\beta$ since, for example, if $\alpha +\beta $ is irrational then $m=n=1$ will do and if $\alpha +\beta $ is rational, then $m=2, n=1$ suffices.  Finally, the maps
\[
S_{\textbf{i}_1}^m \circ S_{\textbf{i}_2}^n \qquad  \text{and} \qquad S_{\textbf{i}_2}^n \circ S_{\textbf{i}_1}^m
\]
satisfy the requirements of the lemma, both rotating by $2 \pi (m \alpha + n \beta)  \, (\text{mod} \, 2 \pi)$ and contracting by $c = c_{\textbf{i}_1}^m c_{\textbf{i}_2}^n \in (0,1)$.
\end{proof}

In light of this lemma, and the fact that Assouad dimension is monotone, it suffices to prove the dense rotations case of Theorem \ref{assouadprojection} for $F$ generated by similitudes of the form given in the lemma. As usual, we assume $F \subseteq [0,1]^{2}$. We wish to reserve ``$c$" for a small constant and ``$t_{i}$" for a real number, so for the mappings $S_{j}$ we write
\begin{displaymath}
 S_{j}(x) = \rho O_\alpha x + w_{j}, \qquad j \in \{1,2\},
 \end{displaymath}
where $O_\alpha \in SO(2)$ is rotation by a fixed $\alpha \notin \pi \mathbb{Q}$, $\rho \in (0,1)$, and $w_{j} \in \mathbb{R}^2$. We will assume without loss of generality that $w_{1} = 0$, so that $S_{1}(0) = 0$, and $0 \in F$. We will use the following special case of a lemma of Ero\u{g}lu, see \cite[Lemma 2.6]{eroglu}:
\begin{lma}\label{eroglu} Given $N \in \mathbb{N}$ and $\varepsilon > 0$, there exist $N$ distinct finite words $\mathbf{i}_{1},\ldots,\mathbf{i}_{N} \in \{1,2\}^{k}$ of (common) length $k \in \mathbb{N}$ such that the following assertions hold:
\begin{itemize}
\item[(a)] All the rotational components of the similitudes $S_{\mathbf{i}_{1}},\ldots,S_{\mathbf{i}_{N}}$ are within $\varepsilon$ of each other.
\item[(b)] There exists $\theta \in [0,\pi)$ such that for each pair $1 \leq i \leq j \leq N$, there exists a point $x = x_{i,j} \in F$ with $|\pi_{\theta}(S_{\mathbf{i}_{i}}(x)) - \pi_{\theta}(S_{\mathbf{i}_{j}}(x))| \leq \varepsilon \rho^{k}$. 
\end{itemize}
\end{lma}

Two rotational components are said to be at distance $\varepsilon$ from each other, if the corresponding angles defining the rotations are. The previous lemma self-improves to the following corollary, where the ``user" may specify the direction $\theta$, the rotational components of the mappings $S_{\mathbf{i}_{j}}$, and the point $x$:

\begin{cor}\label{erCor} Given $N \in \mathbb{N}$, $\varepsilon > 0$ and $\theta \in [0,\pi)$, there exists some $k \in \mathbb{N}$ and distinct finite words $\mathbf{i}_{1},\ldots,\mathbf{i}_{N} \in \{1,2\}^{k}$ such that the following assertions hold:
\begin{itemize}
\item[(i)] All the rotational components of the similitudes $S_{\mathbf{i}_{1}},\ldots,S_{\mathbf{i}_{N}}$ are within $\varepsilon$ of the identity.
\item[(ii)] For each pair $1 \leq i \leq j \leq N$, we have $|\pi_{\theta}(S_{\mathbf{i}_{i}}(0)) - \pi_{\theta}(S_{\mathbf{i}_{j}}(0))| \leq \varepsilon \rho^{k}$. 
\end{itemize}
\end{cor}

\begin{proof} Apply Ero\u{g}lu's lemma with parameters $N' \in \mathbb{N}$ and $\delta > 0$ (to be specified later) to find $k_{0} \in \mathbb{N}$ and $\mathbf{i}_{1},\ldots,\mathbf{i}_{N} \in \{1,2\}^{k_{0}}$ satisfying (a) and (b) of Lemma \ref{eroglu} for some direction $\theta_{0} \in [0,\pi)$. Thus, the rotational components of the mappings $S_{\mathbf{i}_{j}}$ are within $\delta$ of each other, and for each pair $1 \leq i \leq j \leq N'$, there exists a point $x = x_{i,j} \in F$ such that $|\pi_{\theta_{0}}(S_{\mathbf{i}_{i}}(x)) - \pi_{\theta_{0}}(S_{\mathbf{i}_{j}}(x))| \leq \delta \rho^{k_{0}}$.

We first wish to replace $\theta_{0}$ by $\theta$, which is very easy: by the irrationality of $\alpha/\pi$, a suitable choice of $\mathbf{j}_{L} := (1,\ldots,1) \in \{1,2\}^{L}$, with $L \lesssim_{\alpha} 1/\delta$, ensures that 
\begin{displaymath} |\pi_{\theta}(S_{\mathbf{j}_{L}\mathbf{i}_{i}}(x)) - \pi_{\theta}(S_{\mathbf{j}_{L}\mathbf{i}_{j}}(x))| \leq 2\delta \rho^{k_{0} + L}. \end{displaymath}
Since all the rotational components of $S_{\mathbf{i}_{j}}$ were $\delta$-close to each other, the same clearly holds for $S_{\mathbf{j}_{L}\mathbf{i}_{j}}$.   

Next, we wish to replace $S_{\mathbf{j}_{L}\mathbf{i}_{j}}$ by something with rotational component close to the identity, and $x$ by $0$. To this end, observe that all the projections $\pi_{\theta}(S_{\mathbf{j}_{L}\mathbf{i}_{j}}(F))$, $1 \leq j \leq N$, are contained in a single interval of length $5\rho^{k_{0} + L}$, since each individual such projection is contained in an interval of length $2\rho^{k_{0} + L}$, and these intervals are at distance $\leq 2\delta \rho^{k_{0} + L} \leq \rho^{k_{0} + L}$ from each other. Now, for a certain $M \in \mathbb{N}$ (to be specified shortly) and $\mathbf{j}_{M} := (1,\ldots,1) \in \{1,2\}^{M}$, consider the points
\[
t_{j} \, := \,  \pi_{\theta}(S_{\mathbf{j}_{L}\mathbf{i}_{j}}(S_{\mathbf{j}_{M}}(0))) \in \pi_{\theta}(S_{\mathbf{j}_{L}\mathbf{i}_{j}}(F)), \qquad 1 \leq j \leq N'.
\]
There are $N'$ such points, all contained in a single interval of length $5\rho^{k_{0} + L}$. Thus, if $N'$ is large enough, depending only on $\delta$, we can find, by the pigeonhole principle, a subset of cardinality $N$ that is contained in a single interval of length $\delta \rho^{k_{0} + L}$. Without loss of generality, assume that this subset is $\{t_{1},\ldots,t_{N}\}$. Now, if $M$ is chosen suitably, with $M \lesssim_{\alpha} 1/\varepsilon$, the rotational components of the similitudes $S_{\mathbf{j}_{L}\mathbf{i}_{j}\mathbf{j}_{M}}$ are within $\varepsilon$ of identity. Moreover,
\begin{displaymath}
 |\pi_{\theta}(S_{\mathbf{j}_{L}\mathbf{i}_{j}\mathbf{j}_{M}}(0)) - \pi_{\theta}(S_{\mathbf{j}_{L}\mathbf{i}_{j + 1}\mathbf{j}_{M}}(0))| \leq \delta\rho^{k_{0} + L} \leq \varepsilon\rho^{k_{0} + L + M} 
 \end{displaymath}
for all $1 \leq j < N$, if $\delta$ was chosen to be smaller than $\varepsilon \rho^{M}$ to begin with. This completes the proof of the corollary. \end{proof}

We also need the following simple and well-known geometric fact:
\begin{lma}\label{geoLemma} For $x \in \mathbb{R}^{2}$, and any angles $\theta_{1},\theta_{2} \in [0,\pi)$, we have
\begin{displaymath} |\pi_{\theta_{1}}(x) - \pi_{\theta_{2}}(x)| \leq |x||\theta_{1} - \theta_{2}|. \end{displaymath}
Moreover, if $|\pi_{\theta_{j}}(x)| \leq |x|/100$ for $j \in \{1,2\}$, then also  
\begin{displaymath} |\pi_{\theta_{1}}(x) - \pi_{\theta_{2}}(x)| \geq |x||\theta_{1} - \theta_{2}|/100. \end{displaymath}
\end{lma}

\begin{proof} Assume without loss of generality that $x = (0,y)$. Then
\begin{displaymath} |\pi_{\theta_{1}}(x) - \pi_{\theta_{2}}(x)| = |y \sin \theta_{1} - y \sin \theta_{2}| = |y||\sin \theta_{1} - \sin \theta_{2}| \leq |x||\theta_{1} - \theta_{2}| \end{displaymath}
by the mean-value theorem. If $|y||\sin \theta_{j}| = |\pi_{\theta_{j}}(x)| \leq |x|/100 = |y|/100$, then $\theta_{1},\theta_{2}$ are rather close to zero, and the inequality can be reversed up to a multiplicative constant; $1/100$ is certainly on the safe side.  \end{proof}

We define a continuous strictly increasing auxiliary function $\psi: [1,\infty) \to [0,1)$ by $\psi(x) := c_{1} \arctan(x) + c_{2}$, where $c_{1},c_{2}$ are chosen so that $\psi(1) = 0$, $0 < \psi(x) < 1$ for all $x \in (1, \infty)$, and $\psi(x) \nearrow 1$ as $x \to \infty$. 

The main chore on route to our theorem is to establish the following lemma by induction:
\begin{lma}\label{indClaim} Given any $r \in (0,\rho]$ and $M \in \mathbb{N}$ we can find a direction $\theta = \theta(r,M) \in [0,\pi)$ with the following property. There exists $\tau < r$ such that the projection $\pi_{\theta}F$ contains a subset $\{t_{1},\ldots,t_{M}\}$ with the property that
\begin{displaymath} \tau(2 - \psi(M)) \leq t_{j + 1} - t_{j} \leq \tau(2 + \psi(M)), \qquad 1 \leq j < M. \end{displaymath}
\end{lma}

\begin{proof}
The case $M = 2$ is clear, so assume that the lemma has been proven for some $M \geq 2$, and all $r > 0$. Fix $r > 0$, and choose $\theta_{0} \in [0,\pi)$ such that the projection $\pi_{\theta_{0}}(F)$ contains $T_{M} := \{t_{1},\ldots,t_{M}\}$ with $\tau(2 - \psi(M)) \leq t_{j + 1} - t_{j} \leq \tau(2 + \psi(M))$ for some $\tau < r$ and all $1 \leq j <M$. Find $M$ points $x_{1},\ldots,x_{M} \in F$ such that $\pi_{\theta_{0}}(x_{j}) = t_{j}$. Thus,
\begin{equation}\label{form1} \tau(2 - \psi(M)) \leq \pi_{\theta_{0}}(x_{j + 1}) - \pi_{\theta_{0}}(x_{j}) \leq \tau(2 + \psi(M)) \end{equation}
for all $1 \leq j < M$. Without loss of generality, we may assume that $0 \in F$, so also $S_{\mathbf{i}}(0) \in F$ for all finite words $\mathbf{i} \in \{1,2\}^*$. 

Next, apply Corollary \ref{erCor} with $\theta = \theta_{0}$, $N \geq C/\tau$ and $\varepsilon = c\tau$, where $C \geq 1$ and $c > 0$ are large and small constants, respectively, depending on $M$ and to be specified later. We obtain finite words $\mathbf{i}_{1},\ldots,\mathbf{i}_{N}$ of equal length $k \in \mathbb{N}$ such that the rotational components of $S_{\mathbf{i}_{j}}$ are within $\varepsilon$ of the identity, and 
\begin{equation}\label{form5} |\pi_{\theta_{0}}(S_{\mathbf{i}_{j}}(0)) - \pi_{\theta_{0}}(S_{\mathbf{i}_{j + 1}}(0))| < \varepsilon \rho^{k} \end{equation}
for $1 \leq j < N$. It follows that the mutual distance between two of the sets $S_{\mathbf{i}_{j}}(F)$ is at least $N\rho^{k}/10$, and, relabeling the sets if necessary,  we assume that 
\begin{equation}\label{form4} |S_{\mathbf{i}_{1}}(F) - S_{\mathbf{i}_{2}}(F)| \geq N\rho^{k}/10. \end{equation}

Let $\theta \in [0,\pi)$ be an angle close to $\theta_{0}$, to be specified later, and consider the $(M + 1)$-element set
\begin{displaymath} T_{M + 1} := \{\pi_{\theta}(S_{\mathbf{i}_{2}}(0))\} \cup \{\pi_{\theta}(S_{\mathbf{i}_{1}}(x_{j})) : 1 \leq j \leq M\}, \end{displaymath}
see Figure \ref{fig1}. 
\begin{figure}[H]
\begin{center}
\includegraphics[scale = 0.6]{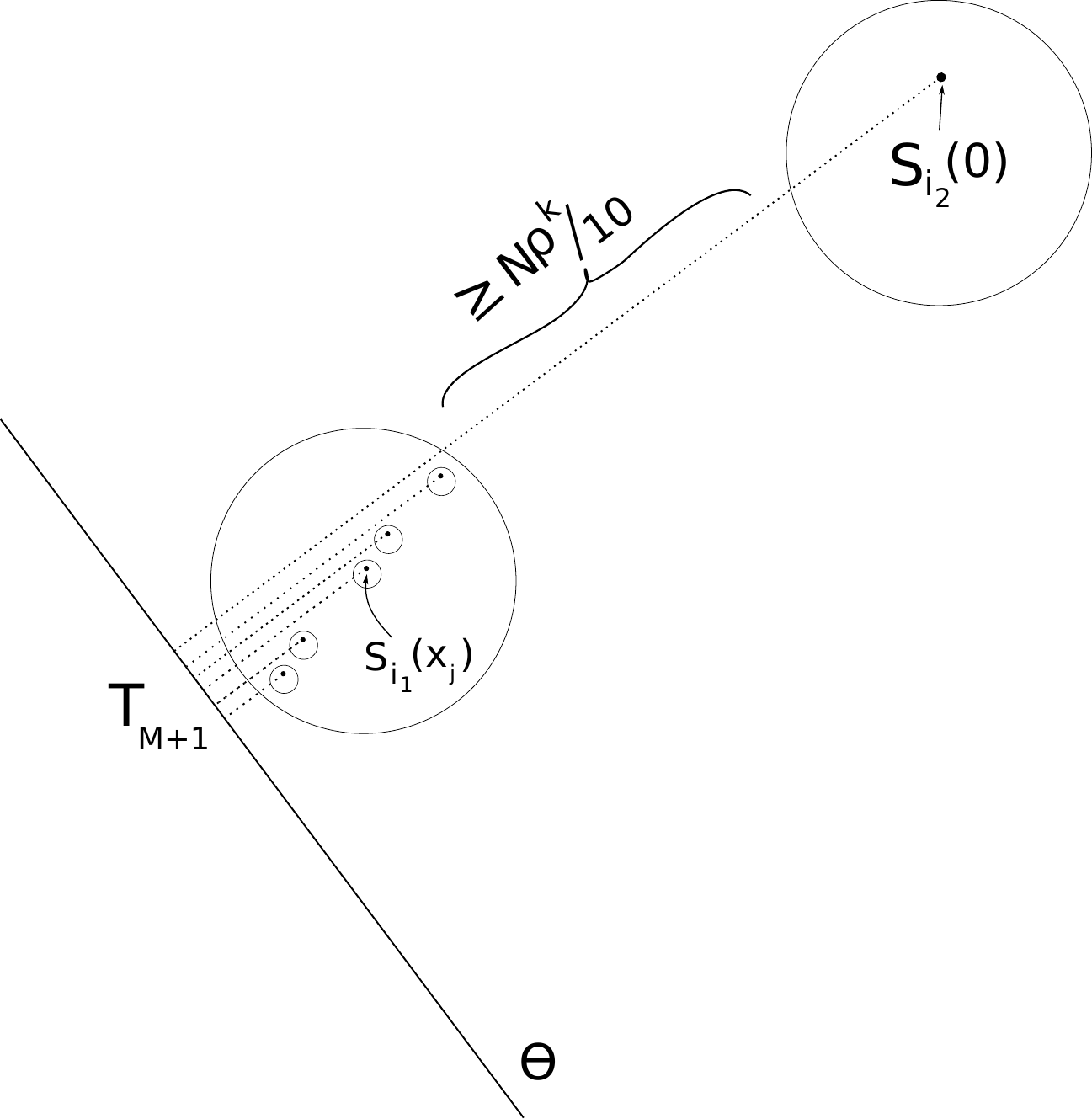}
\caption{Finding $M + 1$ correctly spaced points in the projection $\pi_{\theta} F$.}\label{fig1}
\end{center}
\end{figure}

Clearly $T_{M + 1} \subset \pi_{\theta}(F)$, so it suffices to verify that, for a suitable choice of $\theta$, the elements in $T_{M + 1}$ can be ordered so that any consecutive points have distance between 
\begin{displaymath} \rho^{k}\tau(2 - \psi(M + 1)) \quad \text{and} \quad \rho^{k}\tau(2 + \psi(M + 1)). \end{displaymath} 

To this end, we first consider the points $\pi_{\theta}(S_{\mathbf{i}_{1}}(x_{j}))$, for $1 \leq j \leq M$. Write $S_{\mathbf{i}_{1}}(x) = \rho^{k}Ox + w_{\mathbf{i}_{1}}$, where $O = O_\alpha^k$ is within $\varepsilon = c\tau$ of the identity by assumption. Then, write $e_{0} := (\cos \theta_{0},\sin \theta_{0})$ and $e_{1} = (\cos \theta_{1},\sin \theta_{1}) := Oe_{0}$, so that $|\theta_{1} - \theta_{0}| \leq c\tau$. With this notation,
\begin{align} \pi_{\theta_{1}}(S_{\mathbf{i}_{1}}(x_{j + 1})) - \pi_{\theta_{1}}(S_{\mathbf{i}_{1}}(x_{j})) & = \rho^{k}\big(O(x_{j + 1}) \cdot e_{1} - O(x_{j}) \cdot e_{1}\big) \notag\\
& = \rho^{k}\big(x_{j + 1} \cdot e_{0} - x_{j} \cdot e_{0}\big)\notag\\
& = \rho^{k}\big(\pi_{\theta_{0}}(x_{j + 1}) - \pi_{\theta_{0}}(x_{j})\big)\notag \\
&\label{form2} \in [\rho^{k}\tau(2 - \psi(M)),\rho^{k}\tau(2 + \psi(M))]. \end{align}
by \eqref{form1}, for all $1 \leq j < M$. In fact, the same holds if we replace $\theta_{1}$ by any angle $\theta \in [0,\pi)$ satisfying $|\theta - \theta_{0}| \leq c\tau$: just observe that $|S_{\mathbf{i}_{1}}(x_{j}) - S_{\mathbf{i}_{1}}(x_{j + 1})| \leq \diam(S_{\mathbf{i}_{1}}(F)) \leq 2\rho^{k}$, and then infer from Lemma \ref{geoLemma} that
\begin{displaymath} |\pi_{\theta_{1}}(S_{\mathbf{i}_{1}}(x_{j}) - S_{\mathbf{i}_{1}}(x_{j + 1})) - \pi_{\theta}(S_{\mathbf{i}_{1}}(x_{j}) - S_{\mathbf{i}_{1}}(x_{j + 1}))| \lesssim \rho^{k}|\theta_{1} - \theta| \leq c\tau \rho^{k}. \end{displaymath}
If $c > 0$ is small enough, depending on the difference $\psi(M + 1) - \psi(M)$, combined with \eqref{form2} this proves that
\begin{equation}\label{form3} \rho^{k}\tau(2 - \psi(M + 1)) \leq \pi_{\theta}(S_{\mathbf{i}_{1}}(x_{j + 1})) - \pi_{\theta}(S_{\mathbf{i}_{1}}(x_{j})) \leq \rho^{k}\tau(2 + \psi(M + 1)) \end{equation}
for $1 \leq j \leq M$, as long as $|\theta - \theta_{0}| \leq c\tau$.

Now we have shown that $T_{M + 1}$ contains at least $M$ elements, which lie in the same order as the points $t_{j}$, and such that the distance between consecutive points is within the correct range. To increase this number to $M + 1$, we need to choose $\theta$ so that that $\pi_{\theta}(S_{\mathbf{i}_{2}}(0))$ can act as the $(M + 1)^{\mathrm{th}}$ element. There is not much choice: either $\pi_{\theta}(S_{\mathbf{i}_{2}}(0))$ should lie at distance $2\tau\rho^{k}$ left from $\pi_{\theta}(S_{\mathbf{i}_{1}}(x_{1}))$, or at distance $2\tau \rho^{k}$ right from $\pi_{\theta}(S_{\mathbf{i}_{1}}(x_{M}))$. We make the first choice, which means the question becomes: can we guarantee that $|\theta - \theta_{0}| \leq c\tau$ (in order to maintain \eqref{form3})?

To answer this affirmatively, observe that \eqref{form5} implies $|\pi_{\theta_{0}}(S_{\mathbf{i}_{2}}(0)) - \pi_{\theta_{0}}(S_{\mathbf{i}_{1}}(x_{1}))| \leq 5 \rho^{k}$, and furthermore
\begin{displaymath} |S_{\mathbf{i}_{2}}(0) - S_{\mathbf{i}_{1}}(x_{1})| \geq N\rho^{k}/10 \end{displaymath}
by the choice of $\mathbf{i}_{1},\mathbf{i}_{2}$, and \eqref{form4}. Now, let $\theta \in [0,\pi)$ be an angle such that 
\begin{displaymath}  \pi_{\theta}(S_{\mathbf{i}_{1}}(x_{1})) - \pi_{\theta}(S_{\mathbf{i}_{2}}(0)) = 2\tau \rho^{k}. \end{displaymath} 
By Lemma \ref{geoLemma} and the triangle inequality,
\begin{align*} N\rho^{k}|\theta - \theta_{0}| \lesssim |\pi_{\theta_{0}}(S_{\mathbf{i}_{2}}(0) - S_{\mathbf{i}_{1}}(x_{1})) - \pi_{\theta}(S_{\mathbf{i}_{2}}(0) - S_{\mathbf{i}_{1}}(x_{1}))| \leq (2\tau + 5)\rho^{k}.  \end{align*} 
Recalling that $N \geq C/\tau$ for a large constant $C$, we can force $|\theta - \theta_{0}| \leq c\tau$, as required. Thus, the set $T_{M + 1}$ has the desired properties, and the inductive proof of Lemma \ref{indClaim} is complete.
\end{proof}

 The next simple observation says that Lemma \ref{indClaim} self-improves: the existence of one $\theta$ implies a similar statement for every $\theta$.
\begin{cor} \label{improvedclaim}
 Fix $\theta \in [0,\pi)$, $r \in (0,\rho]$ and $M \in \mathbb{N}$. Then, for some $\tau < r$, the projection $\pi_{\theta}(F)$ contains a set $\{t_{1},\ldots,t_{M}\}$ such that $\tau/2 \leq |t_{j} - t_{j + 1}| \leq 4\tau$ for all $1 \leq j \leq M$.
\end{cor}

\begin{proof}
According to Lemma \ref{indClaim}, for a suitable angle $\theta_{0} \in [0,\pi)$, the projection $\pi_{\theta_{0}}(F)$ contains a subset $\{t^{0}_{1},\ldots,t^{0}_{M}\}$ such that $\tau \leq |t_{j}^{0} - t^{0}_{j + 1}| \leq 3\tau$ for $1 \leq j \leq M$. Write $\mathbf{j}_{L} = (1,\ldots,1) \in \{1,2\}^{L}$. Then, for a certain $\theta_{L} \in [0,\pi)$ depending on the rotation parameter $\alpha$ and $L \in \mathbb{N}$, it follows that the projection $\pi_{\theta_{L}}(S_{\mathbf{j}_{L}}(F))$ contains $M$ points $\{t_{1}^{L},\ldots,t_{M}^{L}\}$ such that $\rho^{L}\tau \leq |t_{j}^{L} - t_{j + 1}^{L}| \leq 3\rho^{L}\tau$. The sequence $(\theta_{L})_{L \in \mathbb{N}}$ is asymptotically dense in $[0,\pi)$, so $|\theta_{L} - \theta| \leq \tau/100$ for some $L \in \mathbb{N}$. Then, writing $t_{j}^{L} = \pi_{\theta_{L}}(S_{\mathbf{j}_{L}}(x_{j}))$ for some $x_{j} \in F$, it follows from Lemma \ref{geoLemma} that
\begin{displaymath} |\pi_{\theta}(S_{\mathbf{j}_{L}}(x_{j}) - S_{\mathbf{j}_{L}}(x_{j + 1})) - \pi_{\theta_{L}}(S_{\mathbf{j}_{L}}(x_{j}) - S_{\mathbf{j}_{L}}(x_{j + 1}))| \leq \diam(S_{\mathbf{j}_{L}}(F))|\theta - \theta_{L}| \leq \rho^{L}\tau/50 \end{displaymath}
for $1 \leq j < M$. This proves the result for the points $t_{j} := \pi_{\theta}(S_{\mathbf{j}_{L}}(x_{j}))$.
\end{proof}

The dense rotations case of Theorem \ref{assouadprojection} now follows immediately. Fix $\theta \in [0,\pi)$ and observe that, by Corollary \ref{improvedclaim}, we can find arbitrarily large $M \in \mathbb{N}$ such that 
\[
N_{\tau/4}\big(B(x,4M\tau) \cap \pi_{\theta}(F)\big) \, \geq \, M \, =  \, \frac{1}{16} \left(\frac{4M\tau}{\tau/4}\right)^{1}
\]
for some $\tau > 0$ and $x \in \pi_\theta F$. This yields $\dim_\text{A} \pi_{\theta} F = 1$, as required.

\section{Proof of Theorem \ref{assouadprojectionexamples}} \label{assouadprojectionexamplesproof}

This theorem is proved using a simple subsystem trick and adapting an example of Bandt and Graf \cite[Section 2 (5)]{BandtGraf}.  First observe that the dense rotations case is taken care of by Theorem \ref{assouadprojection} and so we only need to deal with the discrete rotations case.

\begin{lma}
Let $F\subseteq [0,1]^2$ be a self-similar set which is not contained in a line and suppose that the group generated by $\{O_i\}_{i \in \mathcal{I}}$ is discrete.  Then $F$ contains a self-similar set $E$ which is the attractor of an IFS consisting of three similarities, all of which have the same contraction ratio, no rotational or reflectional component (i.e., trivial orthogonal part) and such that the three fixed points are not collinear.
\end{lma}

\begin{proof}
Again, this lemma is almost trivial since we are not concerned with losing any dimension.  Since $F$ is not contained in a line, we may choose finite words $\textbf{i}_1, \textbf{i}_2, \textbf{i}_3 \in \mathcal{I}^*$ such that the fixed points of $S_{\textbf{i}_1}, S_{\textbf{i}_2}, S_{\textbf{i}_3}$ are not collinear.  Each of these maps has an orthogonal component with finite order and so taking $k$ to be the lowest common multiple of these orders, the maps $S_{\textbf{i}_1}^k, S_{\textbf{i}_2}^k, S_{\textbf{i}_3}^k$ all have trivial orthogonal component. Since the fixed points of these three maps are not collinear, we may choose $m$ sufficiently large to guarantee that there are no triples $(x,y,z) \in S_{\textbf{i}_1}^{mk}(F) \times  S_{\textbf{i}_2}^{mk}(F) \times  S_{\textbf{i}_3}^{mk}(F)$ such that $x,y,z$ are collinear.  Since similarity ratios are multiplicative, the IFS consisting of the three maps
\begin{align*}
& S_{\textbf{i}_1}^{mk} \circ S_{\textbf{i}_2}^{mk} \circ S_{\textbf{i}_3}^{mk} \\
&S_{\textbf{i}_2}^{mk} \circ S_{\textbf{i}_1}^{mk} \circ S_{\textbf{i}_3}^{mk} \\
& S_{\textbf{i}_3}^{mk} \circ S_{\textbf{i}_1}^{mk} \circ S_{\textbf{i}_2}^{mk} 
\end{align*}
all have the same contraction ratios, trivial orthogonal components, and their fixed points are not collinear.  This completes the proof.
\end{proof}

Since Assouad dimension is monotone, we only need to prove the result for self-similar sets of the same form as $E$ from the above lemma.  Let $c \in (0,1)$ be the common similarity ratio and observe that for any $t \geq 0$ we can choose $\theta \in [0,\pi)$ such that $\pi_\theta E$ is an affinely scaled copy of the attractor $E_t$ of the IFS consisting of the maps $S_1, S_2, S_3$ acting on the line defined by
\[
S_1(x) = cx, \qquad S_2(x) = cx+1, \qquad S_3(x) = cx+t.
\]
Thus to complete the proof we need to prove that for some $t \geq 0$, this IFS fails the WSP.  This is a simple adaptation of the example considered by Bandt and Graf in \cite[Section 2 (5)]{BandtGraf}, but we include the details for completeness.  Choose
\[
t = \sum_{k=0}^\infty c^{2^k}
\]
and observe that the maps $S_\textbf{i}^{-1} \circ S_\textbf{j}$ where $\textbf{i}, \textbf{j} \in \{1,2,3\}^n$ are precisely maps of the form
\[
x \to x + \sum_{k=1}^{n} c^{-k} a_k
\]
for any sequence $a_k$ over $\{0, \pm t, \pm 1, \pm (1-t)\}$.  Let $n=2^m$ for some large $m \in \mathbb{N}$ and choose the sequence $\{a_k\}_{k=1}^{n}$ by applying the rule:
\[
 a_k =\left\{\begin{array}{cl}
-1 & \text{if $k = 2^m-2^l$ for some $l=0, \dots, m-1$} \\
t & \text{if $k=n$} \\
0 & \text{otherwise}
\end{array}\right.
\]
This means that for any $m$ we can find maps of the form $S_\textbf{i}^{-1} \circ S_\textbf{j}$ equal to
\[
x \to x + c^{-2^m} t -  \sum_{l=0}^{m-1} c^{2^l-2^m} = x+\sum_{l=m}^{\infty} c^{2^l-2^m}.
\]
Since
\[
0< \sum_{l=m}^{\infty} c^{2^l-2^m} \to 0
\]
as $m \to \infty$, the IFS fails the WSP. We may then apply \cite[Theorem 1.3]{Fraseretal}, or Theorem \ref{GDassouad}, to deduce that $\dim_\text{A} E_t = 1$ and thus there is some $\theta$ for which $\dim_\text{A}\pi_\theta E = 1$.

\section{Proof of Theorem \ref{assouadconstant}} \label{assouadconstantproof}

One of the most effective ways to bound the Assouad dimension of a set from below is to construct \emph{weak tangents}.  This approach was introduced by Mackay and Tyson \cite{mackaytyson}, but the minor adaptation we state and use here was proved in \cite[Proposition 3.7]{Fraseretal}.  First we need a suitable notion of convergence for compact sets, which is given by the Hausdorff metric.  Let $\mathcal{K}(\mathbb{R}^d)$ denote the set of all compact subsets of  $\mathbb{R}^d$, which is a complete metric space when equipped with the Hausdorff metric $d_\mathcal{H}$ defined by
\[
d_\mathcal{H} (A,B) = \max \{ \rho_\mathcal{H} (A,B) , \rho_\mathcal{H} (B,A) \}
\]
where $ \rho_\mathcal{H}$ is defined by
\[
\rho_\mathcal{H} (A,B)  = \sup_{a \in A} \inf_{b \in B} \lvert a-b \rvert,
\]
i.e., the infimal $\delta\geq 0$ such that $A$ is contained in the $\delta$-neighbourhood of $B$.
\begin{prop}{ \em \cite[Proposition 3.7]{Fraseretal}.}\label{pseudoweaktangent}
Let $F, \hat{F} \in \mathcal{K}(\mathbb{R}^d)$ and suppose there exists a sequence of similarity maps $T_k:\mathbb{R}^d \to \mathbb{R}^d $ such that $\rho_\mathcal{H} (\hat{F},T_k(F)) \to 0$ as $k \to \infty$.  Then $\dim_\text{\emph{A}} F \geq \dim_\text{\emph{A}} \hat{F}$. 
\end{prop}

 The set $\hat{F}$ in Proposition \ref{pseudoweaktangent} is called a \emph{weak pseudo-tangent} to $F$.  Note that the insertion of the word `pseudo' in this definition refers to the fact we use $\rho_\mathcal{H}$ instead of the Hausdorff metric used by Mackay and Tyson.  The advantage of this approach is that one only needs a subset of $T_k(F)$ to get close to $\hat{F}$.  This is useful when dealing with overlaps, as we will need to do here.

\begin{lma} \label{tangentsall}
Fix $k,d \in \mathbb{N}$ with $0<k<d$.  Then for all $\pi_1, \pi_2 \in G_{d,k}$, the set $\pi_2 F$ is a weak pseudo-tangent to $\pi_1 F$.
\end{lma}

\begin{proof}
Fix  $\pi_1, \pi_2 \in G_{d,k}$ and let $\varepsilon>0$.   Also, let $d_G$ denote the natural metric on $G_{d,k}$ induced by an appropriate operator norm.  Observe that the map from $(G_{d,k}, d_G)$ to $(\mathcal{K}(\mathbb{R}^k), d_{\mathcal{H}})$ defined by $\pi \mapsto \pi F$ is continuous and so we may choose $\delta>0$ such that for $\pi \in G_{d,k}$
\begin{equation} \label{222}
d_G( \pi_2, \pi ) \leq \delta \ \Rightarrow \ d_{\mathcal{H}}( \pi_2F, \pi F) \leq \varepsilon.
\end{equation}
Since the group generated by $\{O_i\}_{i \in \mathcal{I}}$ is dense in $SO(d)$ or $O(d)$, we may choose a finite word $\textbf{i} \in \mathcal{I}^*$ such that $O_\textbf{i} \in SO(d)$  and
\begin{equation} \label{111}
d_G(  \pi_2, \pi_1 O_\textbf{i}) \leq \delta.
\end{equation}
Here we consider the right action of $SO(d)$ on $G_{d,k}$ defined by $(\pi O) (x) = \pi( O(x))$.  Consider the set $S_\textbf{i}(F) \subset F$ and the projection $\pi_1 S_\textbf{i}(F) \subset \pi_1 F$.  Let $y_\textbf{i} \in F$ denote the unique fixed point of $S_\textbf{i}$ and let $T:\pi_1 F \to \mathbb{R}$ be defined by
\[
T(x) = c_\textbf{i}^{-1} x + \pi_1(y_\textbf{i})(1-c_\textbf{i}^{-1})
\]
which blows up by the reciprocal of the contraction ratio of $S_\textbf{i}$  around the point $\pi_1(y_\textbf{i})$.  The similarity $T$ was defined in this way to ensure that
\[
T\big( \pi_1 S_\textbf{i}(F)\big) = (\pi_1 O_\textbf{i}) F
\]
which, by (\ref{111}) and (\ref{222}), yields
\[
d_{\mathcal{H}}\big(\pi_2 F ,T( \pi_1 S_\textbf{i}(F) )\big) \leq \varepsilon.
\]
This in turn implies that
\[
\rho_\mathcal{H}\big(\pi_2 F ,T( \pi_1 F) \big) \leq \varepsilon
\]
which means that by taking a sequence of $\varepsilon$'s tending to zero and choosing $T$ in this way, we obtain a sequence of similarity maps $T_k$, such that
\[
\rho_\mathcal{H}\big(\pi_2 F ,T_k( \pi_1 F) \big) \to 0
 \] 
 as $k \to \infty$.  This proves that $\pi_2F$ is a weak pseudo-tangent to $\pi_1F$ and completes the proof.
\end{proof}

The fact that the Assouad dimension of $\pi  F$ takes the same value for all $\pi \in G_{d,k}$ now follows immediately from the Lemma \ref{tangentsall} and Proposition \ref{pseudoweaktangent}.

\vspace{6mm}

\begin{centering}

\textbf{Acknowledgements}

The first named author is grateful to \'Abel Farkas, Alexander Henderson, Eric Olson and James Robinson for many fruitful and interesting discussions on topics related to this work, in particular during the writing of \cite{FarkasFraser, Fraseretal}. He also thanks Sascha Troscheit for earlier discussions about the possibility of a `Falconer Theorem for Assouad dimension'.  The first named author is  supported by a \emph{Leverhulme Trust Research Fellowship} and the second named author is supported by the Academy of Finland through the grant \emph{Restricted families of projections and connections to Kakeya type problems, grant number 274512}.   Finally, we wish to thank an anonymous referee for a very careful reading of the manuscript, and numerous helpful comments.
\end{centering}

\begin{multicols}{2}{

\noindent \emph{Jonathan M. Fraser\\
School of Mathematics and Statistics\\
The University of St Andrews\\
St Andrews, KY16 9SS, Scotland} \\

\noindent  Email: jmf32@st-andrews.ac.uk\\ \\

\noindent \emph{Tuomas Orponen\\
Department of Mathematics and Statistics\\
The University of Helsinki\\
Helsinki, Finland} \\

\noindent  Email: tuomas.orponen@helsinki.fi \\ \\
}

\end{multicols}

\end{document}